\documentclass[fullpage, 11pt]{article}
\usepackage[utf8]{inputenc}
\usepackage{xr}
\usepackage[square,numbers,sort&compress]{natbib}

\usepackage[dvips]{graphicx}
\usepackage{epstopdf}

\usepackage[english]{babel}
\usepackage{amsfonts}
\usepackage{amssymb}
\usepackage{amsmath}
\usepackage{latexsym}
\usepackage{amsthm}
\usepackage{epsfig}
\usepackage{graphicx}
\usepackage{color}
\usepackage{subfigure}
\usepackage{enumerate}

\usepackage[hyphens]{url}
\usepackage{hyperref}  

\RequirePackage{tikz}
\RequirePackage{ifthen}
\usetikzlibrary{arrows}

\newcommand\EquiOneDimPoint[1][\relax]{{\text{%
  \tikz[baseline] { 
    \fill (0,0.7ex) circle (0.3ex); 
    \ifx#1\relax
    \else \node[anchor=south west,inner sep=0pt] at (2.1ex, 0)  {$ #1 $};
    \fi
  }}}
}

\newcommand\EquiOneDimEdge[1][\relax]{{\text{%
  \tikz[baseline] { 
    \draw[>=serif cm,<->, thick] (0,0.7ex) -- (2.0ex, 0.7ex); 
    \ifx#1\relax
    \else \node[anchor=south west,inner sep=0pt] at (2.1ex, 0)  {$ #1 $};
    \fi
  }}}
}

\usepackage{ifpdf}

\newcommand\inputfigeps[1]{\ifpdf
\includegraphics{figures/#1.pdf}
\else
\includegraphics{figures/#1.eps}
\fi
}
\graphicspath{{figures/}}

\newcommand*{\vcenteredhbox}[1]{\begingroup
\setbox0=\hbox{#1}\parbox{\wd0}{\box0}\endgroup}

\DeclareMathOperator    \intr                   {int}

\DeclareMathOperator    \relint         {rel\,int}

\DeclareMathOperator    \verts          {vert}

\DeclareMathOperator    \Aff {Aff}  

\newcommand\ve{\relax} 

\newcommand{\old}[1]{{}}

\newcommand{\bb}{\mathbb}

\newcommand{\R}{\bb R}
\newcommand{\Q}{\bb Q}
\newcommand{\Z}{\bb Z}
\newcommand{\N}{\bb N}
\newcommand{\C}{\bb C}

\newcommand\st{\mid}
\newcommand\bigst{\mathrel{\big|}}

\newtheorem{theorem}{Theorem}[section]
\newtheorem{corollary}[theorem]{Corollary}
\newtheorem{lemma}[theorem]{Lemma}

\newtheorem{assumption}[theorem]{Assumption}

\theoremstyle{definition}
\newtheorem{definition}[theorem]{Definition}
\newtheorem{remark}[theorem]{Remark}

\makeatletter
\def\th@claim{%
  \let\thm@indent\indent 
  \thm@headfont{\itshape}%
  \normalfont 
  \thm@preskip\topsep \divide\thm@preskip2
  \thm@postskip\thm@preskip
}
\makeatother

\makeatletter
\renewenvironment{proof}[1][\proofname]{\par
  \pushQED{\qed}%
  \normalfont \topsep6\p@\@plus6\p@\relax
  \trivlist
  \item[\hskip\labelsep
        \itshape\bfseries
    #1\@addpunct{.}]\ignorespaces
}{%
  \popQED\endtrivlist\@endpefalse
}

\makeatother

\makeatletter
\newcommand\Step[1]{\textbf{Step~#1.}
  \def\@currentlabel{#1}}
\makeatother

\newcommand{\bpi}{\bar \pi}

\newcommand{\I}{\mathcal{P}}
\renewcommand{\P}{\mathcal{P}}
\newcommand{\Iedge}[1][q]{\I_{#1,\edge}}
\newcommand{\Ipoint}[1][q]{\I_{#1,\point}}
\newcommand\edge{\EquiOneDimEdge}
\newcommand\point{\EquiOneDimPoint}

\newcommand{\E}{\mathcal{E}}
\newcommand{\G}{\mathcal{G}}
\renewcommand{\S}{\mathcal{S}}

\newcommand{\B}{B}

\def\st{\mid}

\newcommand\Tau{\Gamma^+}

\newcommand\EqClass[1]{[#1]} 

\makeatletter
\let\bfseries=\undefined
\DeclareRobustCommand\bfseries
{\not@math@alphabet\bfseries\mathbf
  \boldmath\fontseries\bfdefault\selectfont}
\makeatother

\makeatletter
\renewcommand{\pod}[1]
{\allowbreak\mathchoice{\mkern18mu}{\mkern8mu}{\mkern8mu}{\mkern8mu}(#1)}
\makeatother

\title{Equivariant Perturbation in \\Gomory and Johnson's Infinite Group
  Problem.\\[1ex] I. The One-Dimensional Case}
\author{
Amitabh Basu\thanks{Dept.\ of Mathematics, University of California, Davis, 
\texttt{abasu@math.ucdavis.edu}} \and
Robert Hildebrand\thanks{Dept.\ of Mathematics, University of California, Davis,
\texttt{rhildebrand@math.ucdavis.edu}} \and
Matthias K\"oppe\thanks{Dept.\ of Mathematics, University of California, Davis, 
\texttt{mkoeppe@math.ucdavis.edu} }
}

\date{\today}

\usepackage[normalem]{ulem}

\begin{document}

\maketitle
\begin{abstract}
We give an algorithm for testing the extremality of minimal valid functions
for Gomory and Johnson's infinite group problem that are  piecewise linear
(possibly discontinuous) with rational breakpoints. This is the first set of
necessary and sufficient conditions that can be tested algorithmically for
deciding extremality in this important class of minimal valid functions.  
We also present an extreme function that is a piecewise linear function with
some irrational breakpoints, whose extremality follows from a new principle.
\end{abstract}
\section{Introduction}

\paragraph{Cutting planes for mixed integer optimization.}

Cutting planes are a key ingredient in branch and cut algorithms, the
state-of-the-art technology for solving mixed integer optimization problems.
Strong cutting planes for \emph{combinatorial} optimization problems
arise from detailed studies of the convex geometry and polyhedral
combinatorics of the problem, more specifically of the convex hull of the
incidence vectors of feasible solutions.  Thousands of research papers giving
sophisticated problem-specific cutting planes (for instance, most famously,
for the Traveling Salesperson Problem) have appeared since the early 1980s. 

In contrast, the state-of-the-art solvers (both commercial and academic) for
general mixed integer optimization problems are based on extremely simple
principles of generating cutting planes that go back to the 1960s, but whose
numerical effectiveness has only been discovered in the mid-1990s
\cite{balas96gomory}.   

The optimal simplex tableau, describing an optimal solution to
the continuous relaxation of an integer optimization problem
\begin{equation}
  \label{eq:orig-ip}
  \max \{ \, \ve c \cdot\ve x \st A \ve x = \ve b,\ \ve x\in\Z^n_+ \,\}
\end{equation}
takes the form
\begin{equation}
  \ve x_B = A_B^{-1} \ve b\, +\, (- A_B^{-1} A_N) \ve x_N, \quad \ve x_B
  \in \Z_+^B,\ \ve x_N \in \Z_+^N, 
\end{equation}
where the subscripts $B$ and $N$ denote the basic and non-basic parts of the
solution~$\ve x$ and matrix~$A$, respectively. 
The widely used, numerically effective general-purpose cuts such as Gomory's
mixed integer (GMI) cut \citep{gom} are derived by simple integer rounding
principles from a \emph{single row}, corresponding to some basic
variable~$x_i$, of the tableau: 
$$ x_i = -f_i + \sum_{j\in N} r_j x_j, \quad x_i \in \Z_+,\ \ve x_N \in
\Z_+^N.  $$
For example, if a tableau row reads
\begin{displaymath}
  x_1 = -(-\tfrac45) - \tfrac15 x_2 + \tfrac25 x_3 + \tfrac{11}5 x_4,
\end{displaymath}
we can use the GMI formula, which describes a periodic, piecewise linear
function $\pi_{f_i}\colon \R\to\R$, to determine the coefficients of a cutting
plane, one-by-one:
\begin{displaymath}
  \pi_{f_i}(-\tfrac15) = \tfrac14,\quad \pi_{f_i}(\tfrac25) = \tfrac34, \quad \pi_{f_i}(\tfrac{11}5)
  = 1.
\end{displaymath}
Thus, we obtain the (very strong) cutting plane
\begin{equation}\label{eq:example-gmi-cut}
  \tfrac14 x_2 + \tfrac34 x_3 + x_4 \geq 1.
\end{equation}

Unfortunately, the performance of cutting-plane algorithms has stagnated since
the computational breakthroughs of the late 1990s and early 2000s.  

\paragraph{The quest of the effective multi-row cut.}

To meet the challenges of ever more demanding applications, it is necessary to
make effective use of information from several rows of the tableau for
generating cuts.  Finding such effective \emph{multi-row cuts} is one of the most
important open questions in cutting plane theory and in practical mixed-integer
linear optimization. 

The past seven years have seen a revival of so-called \emph{intersection
  cuts}, originally introduced by Balas \cite{bal} in 1971.  This research
trend was started by Andersen et al.\ \cite{alww}, who considered the
relaxation 
$$ \ve x_B = A_B^{-1} \ve b\, +\, (- A_B^{-1} A_N) \ve x_N, \quad \ve x_B \in
{\Z^B},\ \ve x_N \in \R_+^N. $$
In this relaxation, 
\begin{itemize}
\item the basic variables $\ve x_B$ are not restricted to be non-negative, but
  are still required to be integers;
\item the non-basic variables $\ve x_N$ are restricted to be non-negative, but
  are no longer required to be integers.
\end{itemize}
This setup, in which maximal lattice-free convex bodies play a central role,
can be studied using the classical tools of convex geometry and the Geometry
of Numbers and has proved to be a highly fruitful research direction \cite{DBLP:journals/siamjo/AndersenWW09,bbcm,bccz}.
Unfortunately, the recent numerical studies of cutting planes based on these
techniques have been disappointing; only marginal improvements upon the
standard GMI cuts have been obtained \cite{lp}.

\paragraph{Gomory's relaxations revisited.}

We study a different relaxation, the \emph{infinite group
  problem}, which goes back to ``classic'' work by Gomory \cite{gom} in
the 1960s and Gomory--Johnson \cite{infinite, infinite2} in the 1970s.  
It is an elegant infinite dimensional generalization of earlier concepts, 
Gomory's \emph{finite group relaxation} and the closely associated \emph{corner
polyhedron}~\cite{gom}.  Both the finite and the infinite group
problem have played a very important role in the theory of deriving valid
cutting planes for integer programming problems, and thus in the foundational
aspects of integer 
programming.  The problems have attracted renewed attention in the past
decade, with several
recent papers discovering very intriguing structures in these problems, and
connecting with some deep and beautiful areas of mathematics~\cite{kianfar3,
  bccz2, bccz08222222, bhkm, BorCor, cm, 3slope, dey1, dey2, deyRichard,
  tspace, kianfar1, kianfar2}. There remain many significant open
problems which provide fertile grounds for future research.  
A more detailed discussion of the importance of the infinite group problem can be
found in the recent survey by Conforti, Cornu\'ejols and
Zambelli~\cite{corner_survey}.\smallbreak

Gomory's group relaxation is defined as 
\begin{equation}
  \label{eq:group-relax}
  \ve x_B = A_B^{-1} \ve b\, +\, (- A_B^{-1} A_N) \ve x_N, \quad \ve x_B \in
  {\Z^B},\ \ve x_N \in \Z_+^N. 
\end{equation}

This relaxation is stronger than the one by Andersen et al., since the
non-basic variables are required to be non-negative \emph{integers}.

Instead of the full tableau, one can again study just a single row of the
tableau, or a few rows.  In the numerical example above, a single-row group
relaxation reads:
\begin{displaymath}
  x_1 = -(-\tfrac45) - \tfrac15 x_2 + \tfrac25 x_3 + \tfrac{11}5 x_4,\quad
  x_1\in\Z, \ x_2,x_3,x_4\in\Z_+.
\end{displaymath}
This equation can be equivalently written as a group
equation in $\R/\Z$ by reading the equation modulo 1; the variable~$x_1$
disappears in this way:
\begin{displaymath}
  0 \equiv -\tfrac15 + \tfrac45 x_2 + \tfrac25 x_3 + \tfrac{1}5 x_4\pmod1,\ x_2,x_3,x_4\in\Z_+.
\end{displaymath}
The example, arising from a basis of determinant $q=5$ and all-integer problem
data, has tableau data that are multiples of $\frac1q=\frac15$.  By
renaming 
\begin{displaymath}
  x_2 = s(\tfrac45), \quad x_3 = s(\tfrac25),\quad x_4=s(\tfrac15)
\end{displaymath}
and introducing extra variables $s(\frac05)$, $s(\frac35)$ for \emph{every}
possible coefficient that is a multiple of~$\frac1q$, we obtain the
\emph{finite master group relaxation}
\begin{displaymath}
  0 \equiv -\tfrac15 + \tfrac05\cdot s(\tfrac05) + \tfrac15\cdot s(\tfrac15)  + \tfrac25\cdot
  s(\tfrac25) + \tfrac35\cdot s(\tfrac35) + \tfrac45\cdot s(\tfrac45)  \pmod1,\
  s(\tfrac05),\dots,s(\tfrac45)\in\Z_+,
\end{displaymath}
which only depends on the value $f_i = \tfrac15$ and on $q=5$.  
Now we go one
step further and introduce infinitely many new variables $s(r)$ for every $r\in
\R$, viz., 
a function $s\colon \R\to\Z_+$, and obtain the
\emph{infinite group relaxation}
\begin{displaymath}
  0 \equiv -\tfrac15 + \sum_{r \in \R} r \cdot s(r) \pmod1,\quad
  s\colon \R\to\Z_+ \text{ a function of finite support},
\end{displaymath}
which only depends on the value $f_i = \tfrac15$.\smallbreak

\paragraph{Formal definition of the problem.}
More formally, Gomory's group problem~\cite{gom} considers an abelian group $G$, written additively, and studies the set of functions $s \colon G \to \R$ satisfying the following constraints: 
                \begin{gather}
                        \sum_{r \in G} r s(r) \in f + S  \label{GP} \\
                        s(r) \in \mathbb{Z}_+ \ \ \textrm{for all $r \in G$}  \notag\\
                        s \textrm{ has finite support}, \notag
                \end{gather}
where $f$ is a given element in $G$, and $S$ is a subgroup of $G$ (not
necessarily of finite index in~$G$); so $f + S$ is the coset containing the element $f$. 

In particular, we are interested in studying the convex hull $R_f(G,S)$ of all
functions satisfying the constraints in~\eqref{GP}. Observe that $R_f(G,S)$ is
a convex subset of the (possibly infinite-dimensional) vector space
$\mathcal{V}$ of functions $s \colon G \to \R$ with finite support. 

An important case is the infinite group problem, where $G =\R^k$ is
taken to be the group of real vectors under addition, and $S= \Z^k$ is the subgroup of
the integral vectors.  
In this paper, we are considering the one-dimensional
case of the infinite group problem, i.e., $k=1$.  This is an important
stepping stone for the larger goal of effective multi-row cuts.  
We extend our results to the case $k=2$ in 
\cite{basu-hildebrand-koeppe:equivariant-2,bhk-IPCOext}.

\paragraph{Valid inequalities and valid functions.} Any linear inequality in $\mathcal{V}$ is given by a pair $(\pi, \alpha)$ where $\pi$ is a function $\pi\colon G \to \R$ (not necessarily of finite support) and $\alpha \in \R$. The linear inequality is then given by $\sum_{r \in G} \pi(r)s(r) \geq \alpha$; the left-hand side is a finite sum because $s$ has finite support. Such an inequality is called a {\em valid inequality} for $R_f(G,S)$ if $\sum_{r \in G} \pi(r)s(r) \geq \alpha$ for all $s \in R_f(G,S)$.

For historical and technical reasons, we concentrate on those valid
inequalities for which $\pi \geq 0$. This implies that we can choose, after a
scaling, $\alpha = 1$. Thus, we only focus on valid inequalities of the form
$\sum_{r \in G} \pi(r)s(r) \geq 1$ with $\pi \geq 0$. Such functions $\pi$
will be termed {\em valid functions} for $R_f(G,S)$. As pointed out in
\cite{corner_survey}, the non-negativity assumption in the definition of a
valid function might seem artificial at first. Although there exist valid
inequalities $\sum_{r \in \R} \pi(r)s(r) \geq \alpha$ 
for $R_f(\R,\Z)$ such that $\pi(r) < 0$ for some~$r \in \R$, it can be shown that $\pi$ must be non-negative over all \emph{rational} $r \in \Q$. Since data in integer programs are usually rational, it is natural to focus on non-negative valid functions. 

A valid function immediately gives the coefficients of a cutting plane for
Gomory's group relaxation~\eqref{eq:group-relax} and thus the original integer
optimization problem~\eqref{eq:orig-ip}, as the GMI function did in the
example of~\eqref{eq:example-gmi-cut}.

\paragraph{Minimal and extreme functions.} Gomory and Johnson~\cite{infinite,
  infinite2} defined a hierarchy on the set of valid functions, capturing the
strength of the corresponding valid inequalities, which we summarize now. A
valid function $\pi$ for $R_f(G,S)$ is said to be \emph{minimal} for
$R_f(G,S)$ if there is no valid function $\pi' \neq \pi$ such that $\pi'(r)
\le \pi(r)$ for all $r \in G$. 
It is known that for every valid function $\pi$
for $R_f(G,S)$, there exists a minimal valid function $\pi'$ such that $\pi'
\leq \pi$ (see~\cite{bhkm} for a proof in the case when $G=\R^k$,
$S=\Z^k$). 
Also minimal valid functions~$\pi$ satisfy $\pi(r) \leq 1$ for $r\in G$
\cite{corner_survey}. 
Since $s \in R_f(G,S)$ are always non-negative functions, minimal
functions clearly dominate valid functions that are not minimal, making the
latter redundant in the description of $R_f(G,S)$. 

A stronger notion is that
of an {\em extreme function}. A~valid function~$\pi$ is \emph{extreme} for
$R_f(G,S)$ if it cannot be written as a convex combination of two other valid
functions for $R_f(G,S)$, i.e., $\pi = \frac{1}{2}(\pi^1 + \pi^2)$
implies $\pi = \pi^1 = \pi^2$.  It is easy to verify that extreme functions
are minimal. Minimal functions for $R_f(G,S)$ were well characterized by
Gomory for groups $G$ such that $S$ has finite index in~$G$ in~\cite{gom}, and
later for $R_f(\R,\Z)$ by 
Gomory and Johnson~\cite{infinite}. We state these results in a unified
notation in the following theorem, which has the same proof as the theorem
in~\cite{infinite}. 

A function $\pi\colon G \rightarrow \mathbb{R}$ is \emph{subadditive} if
$\pi(x + y) \le \pi(x) + \pi(y)$ for all $x,y \in G$. We say that  $\pi$ is
\emph{symmetric} (or \emph{satisfies the symmetry condition}) if $\pi(x) + \pi(f - x) = 1$ for all $x \in G$.

\begin{theorem}[Gomory \cite{gom}, Gomory and Johnson \cite{infinite}] \label{thm:minimal} Let
  $\pi \colon G \rightarrow \mathbb{R}$ be a non-negative function. Then $\pi$
  is a minimal valid function for $R_f(G,S)$ if and only if $\pi(r) = 0$ for
  all $r\in S$, $\pi$ is subadditive, and $\pi$ satisfies the symmetry
  condition. The first two conditions imply that $\pi$ is periodic modulo~$S$
  (i.e., constant over any coset of $S$).
\end{theorem}
        
\begin{remark} 
  Note that this implies that one can view a minimal
  valid function $\pi$ as a function from $G/S$ to $\R$, and thus
  studying $R_f(G,S)$ is the same as studying $R_f(G/S,0)$.  However, we avoid
  this viewpoint in this paper. 
\end{remark}

All classes of extreme functions described in the literature are piecewise
linear, with the exception of a family of measurable functions constructed
in~\cite{bccz08222222}.  
However, a tight characterization of extreme functions for $R_f(\R,\Z)$ has eluded
researchers for the past four decades now. 

\paragraph{Overview of our main results.}
In this paper, we give an algorithm
for deciding the extremality of piecewise linear functions with rational
breakpoints. This algorithm is then used to prove a simple necessary and
sufficient condition for the extremality of {\em continuous} piecewise linear
minimal functions. To the best of our knowledge, these two results are the
first of their kind; in comparison, the majority of results on extremality in
the literature are very specific sufficient conditions for guaranteeing
extremality~\cite{bhkm, 3slope, dey1, dey2, deyRichard,  tspace,
  Richard-Li-Miller-2009:Approximate-Liftings, Miller-Li-Richard2008}. Moreover, some of the more general sufficient conditions (see, for example, Theorem 6 in~\cite{dey1}) are not algorithmic in the sense that given a particular valid function, it is not possible to test the sufficient condition in a finite number of computations. We give more details below.

\old{Although our results only hold for testing the extremality of piecewise linear functions with rational breakpoints, we believe this is an important step in understanding extreme valid functions. Moreover, continuous piecewise linear functions with rational breakpoints, are arguably the most important class of valid functions, from a practical perspective. This is because we are ultimately interested in finite dimensional integer programs with rational data, and continuous piecewise linear functions with rational breakpoints are the most natural class of valid functions relevant to these programs. We now state our main results.}

For the infinite group problem $R_f(\R,\Z)$, by Theorem~\ref{thm:minimal} a  minimal valid function 
$\pi \colon \R \to \R_+$ is periodic modulo~$\Z$, i.e., periodic with period
$1$. 
We consider piecewise linear functions, possibly discontinuous, which 
have only finitely many breakpoints in the fundamental domain~$[0,1)$.

Our first main result is an algorithm for deciding if a given piecewise linear
function is extreme or not. The proof of this theorem appears in
section~\ref{section:proofs-main}.

\begin{theorem}\label{thm:main}
Consider the following problem.  
\begin{quote}
  Given a minimal 
  valid function $\pi$ for $R_f(\R,\Z)$ that is piecewise
  linear with a set of rational breakpoints with the least common
  denominator~$q$, decide if $\pi$ is extreme or not.
\end{quote}
There exists an algorithm for this problem that takes a number of elementary operations over the reals that is
bounded by a polynomial in $q$. 
\end{theorem}

If we start with any piecewise linear valid function~$\pi$, the first step is to determine if
$\pi$ is minimal.  A minimality test for {\em continuous} piecewise linear
functions was given by Gomory and Johnson (see Theorem 7 in~\cite{tspace}).
In section~\ref{section:minimalityTest} we present a minimality test that
works for discontinuous functions too.\smallskip

Next we investigate the precise relationship between continuous extreme functions for the
infinite group problem and certain finite group problems, i.e., $R_f(G, \Z)$
where $G$ is a discrete subgroup of~$\R$ that contains~$\Z$.  (The word
``finite'' refers to the finite index of $\Z$ in $G$, in other words the
quotient group~$G/\Z$ is finite.) 
A first result in
this direction appeared in \cite{dey1}; we state it in our notation.  
\begin{theorem}[Theorem 6 in~\cite{dey1}]\label{thm:infinite_test}
Let $\pi$ be a
piecewise linear minimal valid function for $R_f(\R,\Z)$ with set $\B$ of rational
breakpoints with the least common denominator~$q$. 
Let $G_{2^n}, n \in \N$ denote the subgroups $ \frac{1}{2^nq} \Z$. Then
$\pi$ is extreme if and only if the restriction $\pi|_{G_{2^n}}$ is extreme
for $R_f(G_{2^n}, \Z)$ for all $n\in \N$.
\end{theorem} 
Clearly the above condition cannot be checked in a finite number of steps and hence
cannot be converted into an algorithm, because it potentially needs to test
infinitely many finite group problems.
In contrast, we prove the following result in section~\ref{section:proofs-main}.
\begin{theorem}\label{thm:finite_group}
  Let $\pi$ be a \emph{continuous} piecewise linear minimal valid
  function for $R_f(\R,\Z)$ with set $\B$ of rational
  breakpoints with the least common denominator~$q$. 
  Let $\hat G = \smash[b]{\frac1{4q} \Z}$.  Then
  $\pi$ is extreme for $R_f(\R, \Z)$ if and only if $\pi|_{\hat G}$ is extreme
  for $R_f(\hat G,\Z)$. 
\end{theorem}
Gomory gives a characterization of extreme valid functions for finite group
problems via a linear program. This provides an alternative algorithm for
testing extremality of {\em continuous} piecewise linear functions with
rational breakpoints, under the light of Theorem~\ref{thm:finite_group}. Of
course, Theorem~\ref{thm:main} is more general as it provides an algorithm
for discontinuous functions also. 


\paragraph{Techniques of this paper.}
The standard technique for showing extremality is to suppose that $\pi =
\frac12(\pi^1+\pi^2)$, where $\pi^1,\pi^2$ are other valid
functions.  One observes the following basic fact:
\begin{lemma}\label{lem:tightness}\label{lem:minimality-of-pi1-pi2}
  Let $\pi$ be minimal, $\pi = \frac12(\pi^1+\pi^2)$ with $\pi^1,\pi^2$
  valid functions.  Then $\pi^1,\pi^2$ are minimal, and all subadditivity
  relations $\pi(x + y) \le \pi(x) + \pi(y)$  
  that are tight for~$\pi$ are also tight for
  $\pi^1,\pi^2$.
\end{lemma}
Then one shows that actually $\pi=\pi^1=\pi^2$ holds. 
The main tool used in the literature for showing this is the so-called Interval
Lemma introduced by Gomory and Johnson in~\cite{tspace}, which we state here
for a more coherent discussion of the new ideas in this paper. 

\begin{lemma}[Interval Lemma~\cite{tspace, bccz08222222}] \label{lem:interval_lemma} Let $\theta \colon \R \rightarrow
\mathbb{R}$ be a function bounded on every bounded interval. Given real numbers $u_1 < u_2$ and $v_1 < v_2$, let $U = [u_1, u_2]$, $V =
[v_1, v_2]$, and $U + V = [u_1 + v_1, u_2 + v_2]$.
If $\theta(u)+\theta(v) = \theta(u+v)$ for every
$u \in U$ and $v \in V$, then there exists $c\in \R$ such that
\begin{alignat*}{2}
  \theta(u)&=\theta(u_1)+c(u-u_1) &\quad& \text{for every $u\in U$,} \\
  \theta(v)&=\theta(v_1)+c(v-v_1) && \text{for every $v\in V$,} \\
  \theta(w)&=\theta(u_1+v_1)+c(w-u_1-v_1) && \text{for every $w\in U+V$.}
\end{alignat*}
\end{lemma}
\begin{remark}
  We remark that the Interval Lemma is a lemma of \emph{real analysis} (the theory of
  functional equations); the hypothesis that the function~$\theta$ is bounded
  on every bounded interval is one of several possible hypotheses to rule out
  certain pathological functions.
\end{remark}

Every proof of extremality in the existing literature employs the Interval
Lemma on proper intervals that satisfy certain additivity conditions to
deduce affine linearity properties that $\pi^1$ and $\pi^2$ share
with~$\pi$.  This is followed by a \emph{linear algebra} argument (explicit in
Gomory--Johnson's proof of the two slope theorem, but implicit in many other
proofs) to establish uniqueness of~$\pi$, and thus its extremality.

Surprisingly, the \emph{arithmetic} (number-theoretic) aspects of the problem
seem to have been largely overlooked, even though they are at the core of the
theory of the closely related \emph{finite} group problem.  This aspect turns
out to be the key for completing the algorithmic classification of extreme
piecewise linear functions.

To capture the relevant arithmetics of the problem, we study finite sets of
additivity relations of the form $\pi(t_i) + \pi(y) = \pi(t_i + y)$ and 
$\pi(x) + \pi(r_i-x) = \pi(r_i)$, where the points $t_i$ and $r_i$ are
certain breakpoints of the function~$\pi$.  They give rise to a useful
abstraction, the \emph{reflection group}~$\Gamma$ generated by the reflections
$\rho_{r_i}\colon x \mapsto r_i-x$ and translations~$\tau_{t_i}\colon y\mapsto
t_i+y$.  

We then study the natural action of the reflection group~$\Gamma$ on the set
of open intervals delimited by the elements of~$G=\frac1q\Z$.  Roughly
speaking, the action of $\Gamma$ transfers the affine linearity established by
the Interval Lemma on some interval~$I$ to the orbit $\Gamma(I)$.  Actually,
this transfer is more delicate, and we have to combine this arithmetic
consideration with a discussion of reachability
.  In the end, the transfer happens within each
connected component of a certain graph.

When the Interval Lemma and this transfer technique establish affine linearity
of $\pi^1, \pi^2$ on all intervals where $\pi$ is affinely linear, we can
proceed with linear algebra techniques to decide extremality
of~$\pi$.  
Otherwise, we show that there is a way to perturb $\pi$ slightly to
construct distinct minimal valid functions $\pi^1 = \pi + \bar\pi$ and $\pi^2
= \pi - \bar\pi$.  
Here
the reflection group~$\Gamma$ gives a blueprint for this perturbation in
the following way.  We use $\Gamma$-equivariant functions, i.e., functions
that are invariant under translation by~$\frac1q$ and odd with respect to
the reflection points~$0,\pm\frac1{2q},\pm\frac1q,\dots$.  Sufficiently small $\Gamma$-equivariant 
functions, again modified by restriction to a certain connected component, are
then suitable perturbation functions~$\bar\pi$.
A particular choice of these functions allows us to prove
Theorem~\ref{thm:finite_group}. 

\paragraph{An interesting irrational function and a complexity conjecture.}

We now discuss the complexity of the algorithm of Theorem~\ref{thm:main}. For
the purpose of this discussion, let us restrict ourselves to a version of the
problem where all input data are rational.  It is an open question whether the
pseudo-polynomial complexity of our algorithm is best possible, or whether
there exists a polynomial-time algorithm, or even a strongly polynomial-time
algorithm, whose running time would only depend on the number of breakpoints
but not on the sizes of the denominators.  We conjecture that the problem (in
a suitable version with all rational input data) is NP-hard in the weak sense.

We believe that the problem is intrinsically arithmetic, so that an
algorithm that is oblivious to the sizes of the denominators is not possible.
To substantiate this, we construct a certain extreme function with
irrational breakpoints (section~\ref{sec:irrational}).  Any nearby
function with rational breakpoints that uses the same construction turns out
to not be extreme.

The proof of extremality of this function 
requires another technique unrelated to the Interval Lemma. Here a
reflection group~$\Gamma$ arises under which every point~$x$ has an
orbit~$\Gamma(x)$ that is dense in~$\R$.  Roughly speaking (ignoring the
reachability issues, which our proof has to discuss), there exists no
non-trivial \emph{continuous} $\Gamma$-equivariant perturbation.  Thus the
function is extreme.

\section{Preliminaries}

\subsection{Polyhedral complexes and piecewise linear functions}

We introduce the notion of polyhedral complexes, which serves two purposes in
our paper.  One, it provides an elegant framework to define discontinuous
piecewise linear functions.  Two, it is a tool for studying subadditivity
relations. 

\begin{definition}
A (locally finite) {\em polyhedral complex} is a collection $\P$ of polyhedra in $\R^k$ such that:
\begin{enumerate}[\rm(i)]
\item if $P \in \P$, then all faces of $P$ (including the empty
  face~$\emptyset$) are in $\P$,
\item the intersection $P \cap Q$ of two polyhedra $P,Q \in \P$ is a face of both $P$ and $Q$,
\item any compact subset of $\R^k$ intersects only finitely many faces in $\P$.
\end{enumerate}
The polyhedral complex~$\P$ is called \emph{periodic modulo~$\Z^k$} if for all
$P\in \P$ and all vectors $\ve t\in\Z^k$, the translated polyhedron~$P + \ve
t$ also lies in~$\P$.
\end{definition}

\subsubsection{Discontinuous piecewise linear functions}

By Theorem~\ref{thm:minimal}, minimal functions are periodic modulo~$\Z$. 
We now give a definition of piecewise linear functions periodic modulo~$\Z$ that allows for
discontinuous functions; see also 
Figure~\ref{figure:discontinuousPiecewisePoints}.

Let $0=x_0 < x_1 < \dots < x_{n-1} < x_n=1$ be a list of (possible)
breakpoints in $[0,1]$.  
We denote by 
\begin{displaymath}
  \B = \{\, x_0 + t, x_1 + t, \dots, x_{n-1}+t\st
t\in\Z\,\}
\end{displaymath}
the set of all breakpoints.  
Define the set of 0-faces to be the collection of singletons,
\begin{displaymath}
  \I_{\B,\point} = \bigl\{\, \{ x \} \st x\in B\,\bigr\},
\end{displaymath}
and the set of 1-faces to be the collection of closed intervals,
\begin{displaymath}
  \I_{\B,\edge} = \bigl\{\, [x_i+t, x_{i+1}+t] \st i=0, \dots, {n-1} \text{
    and } t\in\Z \,\bigr\}. 
\end{displaymath}
Then $\I_{\B} = \{\emptyset\} \cup \I_{\B,\point} \cup \I_{\B,\edge}$ is a
locally finite one-dimensional polyhedral
complex that is periodic modulo~$\Z$.

For each 0-face $I\in
\I_{\B,\point}$, there is a
constant function $\pi_I(x) = b_I$; for each 1-face $I\in
\I_{\B,\edge}$, there is an affine linear function $\pi_I(x) = m_I x
+ b_I$, defined for all $x\in \R$.  A function $\pi\colon \R \to
\R_+$ periodic modulo~$\Z$ is called {\em piecewise linear} if it is given by 
$ \pi(x) = \pi_I(x)$ where $x \in \relint(I) \text{ for some } I \in \I_{\B}$.

Since $\pi$ is periodic modulo~$\Z$, for any $t\in \Z$ we have $\pi_{I+t}(x+t) = \pi_I(x)$ for $I\in
\P_B$ and thus $b_{I+t} = b_I$ for $I \in \I_{\B,\point}$
and $m_{I+t} = m_I$, $b_{I+t} = b_I-m_It$ for $I \in \I_{\B,\edge}$. 

\begin{figure}[t]
\centering
\ifpdf
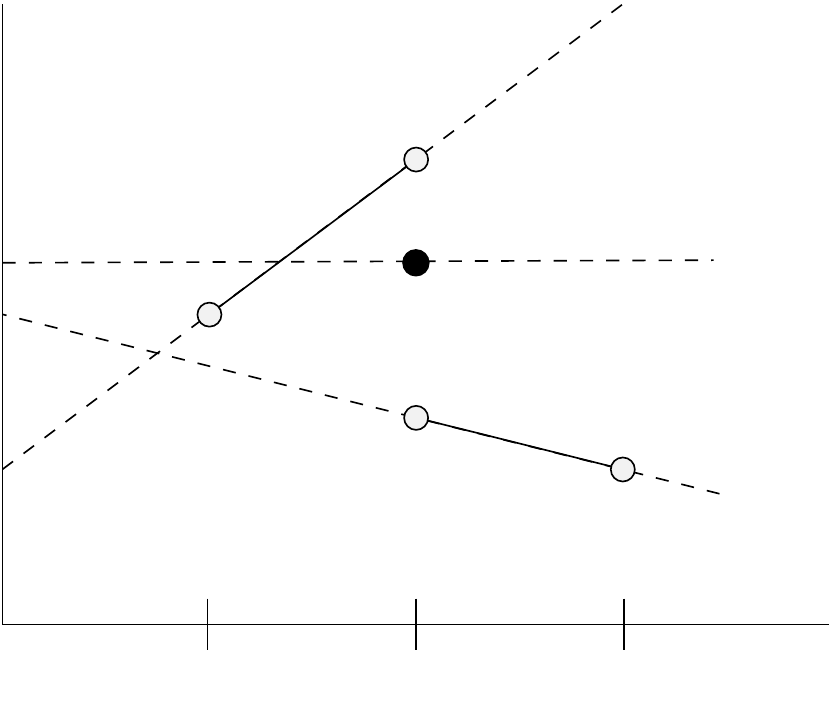
\else
\input{figures/discontinuousPiecewisePointsDashed.pstex_t}
\fi
\caption{A discontinuous piecewise linear function $\pi$ with breakpoints $\B = \{x_0,
  \dots, x_n\}+\Z$ with $0=x_0 < x_1< \dots < x_n=1$.  This figure shows a piecewise linear function $\pi$ on $(x_i,
  x_{i+2})$ where $I= [x_i, x_{i+1}], J = \{x_{i+1}\}, K = [x_{i+1},
  x_{i+2}]$.  Affine linear functions $\pi_I$, $\pi_J$, $\pi_K$ are defined on
  all of~$\R$.  The piecewise linear function~$\pi$ agrees with $\pi_I$ on
  $\relint(I)$, etc.}
\label{figure:discontinuousPiecewisePoints}
\end{figure}


\subsubsection{A two-dimensional polyhedral complex}

The following notation will be used
in the rest of the paper.  
The function $\Delta\pi$ measures the slack in the
subadditivity constraints:
$$\Delta\pi(x,y) = \pi(x) + \pi(y) - \pi(x+
y).$$ 

Let $\Delta\P_\B$ be the two-dimensional polyhedral complex with faces 
$$F = F(I,J,K) = \{\,(x,y) \in \R^{2} \st x \in I, y \in J, x+ y \in K\,\},$$ 
where $I,J,K \in \I_{\B,\point} \cup \I_{\B,\edge}$; see Figure~\ref{fig:non-uniform}.

Since $\P_\B$ is periodic modulo~$\Z$, the complex~$\Delta\P_\B$ is periodic
modulo~$\Z^2$.  By construction, all faces $F$ of the complex are polytopes
(i.e., points, edges, or convex polygons),
and the faces contained in $[0,1]^2$ form a finite system of
representatives, i.e., for every face $F\in\Delta\P_\B$, there exists a face
$F^0\in\Delta\P_\B$ with $F^0 \subseteq [0,1]^2$ and a translation vector
$(s,t)\in \Z^2$ such that $F = F^0 + (s,t)$.  Note that $F^0$ is not always
uniquely determined:  $F = \{(2,2)\}$ is represented by each of the faces
$\{(0,0)\}$, $\{(1,0)\}$, $\{(0,1)\}$, and $\{(1,1)\}$. 

\begin{remark}
  This polyhedral complex was studied by Gomory and Johnson in~\cite{tspace}
  for a different purpose (the so-called merit index).
\end{remark}

\begin{figure}[t]
\begin{center}
\includegraphics[scale = 0.7]{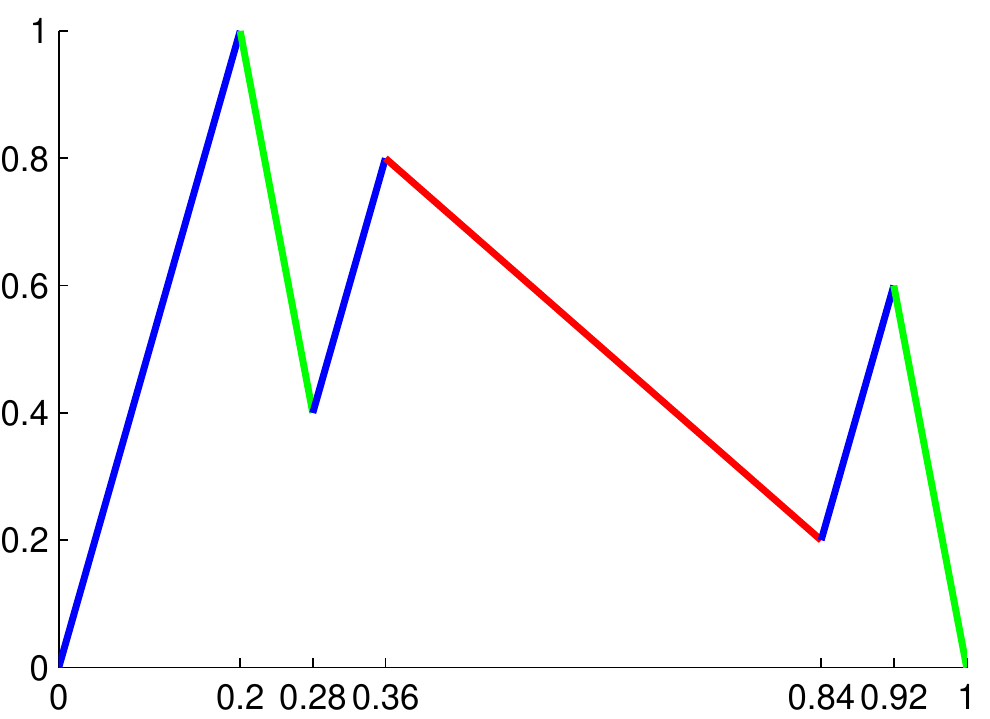}\ 
\includegraphics[scale = 0.7]{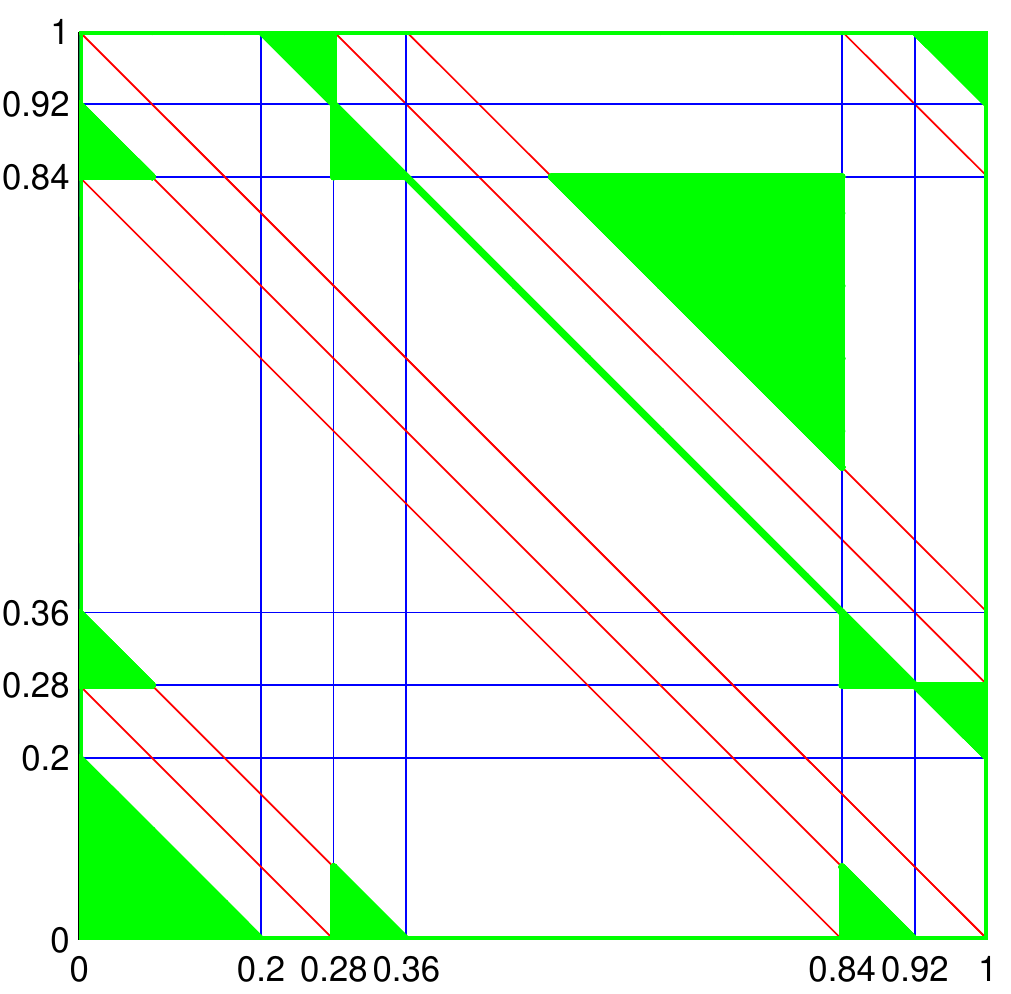}
\end{center}
\caption{A minimal continuous piecewise linear function~$\pi$ with set $B
  = \{0,0.2,\allowbreak 0.28,\allowbreak 0.36,\allowbreak 0.84,0.92,1\}+\Z$ of breakpoints and its two-dimensional
  complex~$\Delta\P_\B$.  Since $\pi$ is periodic modulo~$\Z$, we only show it
  on $[0,1]$.  The complex is defined by vertical lines with $x\in
  B$, horizontal lines with $y \in B$, and diagonal lines with $x+ y \in B$.  
  The faces~$F=F(I,J,K)$ with $\Delta\pi|_{\relint(F)} = 0$ are shaded bright
  green.  Since $\Delta\P_\B$ is periodic modulo~$\Z^2$, we only show it on
  $[0,1]^2$. Together these faces cover all points
  $(x,y)$ with $\Delta\pi(x,y) = 
  0$.  
  (This function~$\pi$ is one version of Gomory and Johnson's
  3-slope function that is constructed in \cite{tspace}.)}
\label{fig:non-uniform}
\end{figure}

Observe that $\Delta\pi|_{\relint(F)}$ is affine; if we introduce the function $$\Delta\pi_F(x,y) =
\pi_I(x) + \pi_J(y) - \pi_K(x+ y)$$ for all $x,y\in\R$, then 
$\Delta\pi(x,y) = \Delta\pi_F(x,y)$ for all $(x,y) \in \relint(F)$.  We will use $\verts(F)$
to denote the set of vertices of the face~$F$.  

For discontinuous functions~$\pi$, using the function $\Delta\pi_F$ for a
face~$F$ instead of $\Delta\pi$ allows us to conveniently express limits to
boundary points of~$F$, in particular to vertices of~$F$, along
paths within~$\relint(F)$. 
\begin{lemma}\label{lemma:face-limits}
  Let $\pi\colon \R\to\R_+$ be a piecewise linear function.  
  Let $F \in \Delta\P_\B$ and let $(u,v) \in F$.  Then
  \begin{equation}
    \Delta\pi_F(u,v) = \lim_{\substack{(x,y) \to (u,v)\\ (x,y) \in \relint(F)}} \Delta\pi(x,y).
  \end{equation}
  In particular, $\Delta\pi_F \geq 0$ on $F$ if and only if $\Delta\pi \geq 0$
  on $\relint(F)$. 
\end{lemma}
\begin{proof}
  This follows from the definition of~$\Delta\pi_F$ and the continuity of the
  affine functions~$\pi_I$, $\pi_J$, and $\pi_K$. 
\end{proof}

\subsection{Finite test for minimality of piecewise linear functions}
\label{section:minimalityTest}

By Theorem~\ref{thm:minimal}, we can test whether a function is minimal by
testing subadditivity and the symmetry condition. 

A finite subadditivity test for {\em continuous} piecewise
linear functions was given by Gomory and Johnson \cite[Theorem
7]{tspace}.\footnote{Note that the word ``minimal'' needs to be replaced by
  ``satisfies the symmetry condition'' throughout the statement of their
  theorem and its proof.}  
It easily extends to the discontinuous case.  We include
a proof in our notation to make the present paper self-contained, but do not
claim novelty.  Richard, Li, and Miller \cite[Theorem
22]{Richard-Li-Miller-2009:Approximate-Liftings}, for example, presented a
superadditivity test for discontinuous piecewise linear functions.

We first show that $f$ must be a breakpoint of any minimal
valid function.

\begin{lemma}
\label{lemma:faVertex}
If $\pi\colon \R\to\R$ is a minimal valid function with breakpoints in~$B$, then $f \in \B$.  
\end{lemma}
\begin{proof}
Since $\pi$ is minimal, we have $0 \leq \pi \leq 1$.
Now, suppose that $f \notin \B$, i.e., $f \in \intr(I)$ for some $I\in
\I_{\B,\edge}$.  Symmetry and the condition that $\pi(0) = 0$ imply that
$\pi(f) = 1$.  Since $\pi \leq 1$ and $\pi$ is affine in $\intr(I)$, it follows that
$\pi \equiv 1$ on $\intr(I)$.  The interval $f-I$ contains the origin in its interior
and by symmetry, $\pi \equiv 0$ on $\intr(f-I)$.  Let $J$ be the largest 
interval containing the origin such that $\pi \equiv 0$ on $J$ and let
$\bar x = \sup\{\, x \in J \,\}$.  Since $J$ is the largest such
interval, and $\bar x \neq \infty$, for every small $\epsilon > 0$, there exists a
point $y \geq \bar x$ such that $\epsilon, y-\epsilon \in J$ and $
\pi(y) >0$.  But then $\pi(\epsilon) + \pi(y - \epsilon) = 0  < \pi(y)$,
which violates subadditivity, and therefore is a contradiction. 
\end{proof}




\begin{theorem}[Minimality test]
\label{minimality-check}
A piecewise linear function $\pi\colon \R
\to\R$ with breakpoints $\B$ that is periodic modulo~$\Z$ is minimal
if and only if the following conditions hold:
\begin{enumerate}
\item $\pi(0) = 0$.

\item Subadditivity test: For all $F \in \Delta\P_\B$
  with $F\subseteq[0,1]^2$, we have that $\Delta\pi_F(u,v) \geq 0$ for all $(u,v) \in\verts(F)$.
\item Symmetry test: $\pi(f) = 1$, and for all $F \in \Delta\P_\B$ with $$F
  \subseteq \bigl\{\,(x,y) \in [0,1]^2 \bigst x + y \equiv f\textstyle\pmod{1}\,\bigr\},$$ 
  we have that $\Delta\pi_F(u,v) = 0$ for all $(u,v)\in \verts(F)$.

\end{enumerate}
\end{theorem}

\begin{proof}
  We use the characterization of minimal functions given by
  Theorem~\ref{thm:minimal}. 


Let $F \in \Delta\P_\B$ with $F\subseteq[0,1]^2$ and let $(u,v) \in \verts(F)$.  Then, by Lemma~\ref{lemma:face-limits},
\begin{equation}
\Delta\pi_F(u,v) = \lim_{\substack{(x,y) \to (u,v)\\ (x,y) \in \relint(F)}} \Delta\pi(x,y).
\end{equation}
Since $\pi$ is subadditive by Theorem~\ref{thm:minimal}, $\Delta\pi(x,y)\geq 0$ for all $x,y \in [0,1]$,
and therefore the limit above is also non-negative.  Suppose $F \subseteq
\{\,(x,y)\in[0,1]^2 \st x+ y \equiv f\pmod1\,\}$.  Since $\pi$ is symmetric
by Theorem~\ref{thm:minimal}, we have  $\pi(x)
+ \pi(f-x) = 1 = \pi(f)$ and therefore $\Delta\pi(x,f-x)=0$ for
$x\in[0,1]$. Since $\pi$ is also periodic modulo~$\Z$ by 
Theorem~\ref{thm:minimal}, we have $\Delta\pi(x,y) = 0$ whenever $x +  y
\equiv f\pmod1$ 
and $x,y \in [0,1]$.  Therefore, the above limit is in fact a limit of zeros,
and $\Delta\pi_F(u,v) = 0$. 

We now show that the stated conditions are sufficient.

For subadditivity, we need to show that $\Delta\pi(x,y) \geq 0$ for all $x,y
\in \R$.  Since $\pi$ is periodic modulo~$1$, it suffices to show that
$\Delta\pi(x,y) \geq 0$ for
$x,y\in[0,1]$. 
Let $x,y \in [0,1]$, then $(x,y) \in \relint(F)$ for some unique $F \in
\Delta\P_\B$ with $F\subseteq[0,1]^2$ and $\Delta\pi(x,y) = \Delta\pi_F(x,y)$.  
Since $\Delta\pi_F(u,v) \geq 0$ for all $(u,v) \in \verts(F)$, by convexity ($\Delta\pi_F$ is affine), $\Delta\pi_F(x,y) \geq 0$.  Therefore $\pi$ is subadditive.  

 Similarly, to show symmetry, because $\pi$ is periodic modulo~$1$, it
 suffices to show that $\Delta\pi(x,y) = 0$ for all $x,y \in [0,1]$ such that
 $x + y \equiv f\pmod1$.  Let $x,y \in [0,1]$ such that $x+ y \equiv f\pmod1$.
 Since $f \in \B$ by Lemma~\ref{lemma:faVertex}, $(x,y) \in \relint(F)$ for
 some $F \in \Delta\P_\B$ for $F \subseteq \{\,(x,y)\in[0,1]^2 \st x+ y \equiv
 f\pmod1\,\}$.   Since $\Delta\pi_F(u,v) = 0$ for all $(u,v) \in \verts(F)$,
 and $\Delta\pi_F$ is affine, it follows that $\Delta\pi(x,y) =
 \Delta\pi_F(x,y) = 0$.   
\end{proof}

 \begin{remark}
   \label{remark:minimality-check-complexity}
This theorem implies that, for $n= |\B\cap[0,1)|$,  
 there are $O(n^2)$ pairs that need to
be checked in order to check subadditivity, and only $O(n)$ points that
need to be evaluated to check for symmetry.  These follow from the fact that
there are $O(n)$ hyperplanes (lines) in the two-dimensional polyhedral
complex~$\Delta\P_B$ in $[0,1]^2$.  Any $0$-dimensional face (vertex) is at the
intersection of at least two hyperplanes (lines), therefore, at most $O(n)
\choose 2$ possibilities.  And any $0$-dimensional face (vertex) corresponding
to a symmetry condition is at the intersection of the hyperplane (line) $x + y
= f$ or $x + y = f+1$
and any other one, therefore, at most $O(n)$ such points.   
\end{remark}%

\subsection{Limit relations}

Let $\pi$ be a minimal valid function that is piecewise linear. 
Suppose $\pi^1$ and $\pi^2$ are minimal valid functions such that  $\pi =
\frac12(\pi^1+\pi^2)$. 
By Lemma~\ref{lem:tightness}, whenever $\pi(x) + \pi(y) =
\pi(x+y)$, the functions $\pi^1$ and $\pi^2$ must also satisfy this
equality relation, that is, $\pi^i(x) + \pi^i(y) = \pi^i(x+ y)$.  Let
$\Delta\pi^i(x,y) = \pi^i(x) + \pi^i(y) - \pi^i(x+ y)$ for $i=1,2$.
Equivalently, whenever $\Delta\pi(x,y) = 0$, we also have $\Delta\pi^i(x,y) =
0$ for $i=1,2$.  This is easy to see since $\pi, \pi^1, \pi^2$ are
subadditive, $\Delta\pi, \Delta\pi^1, \Delta\pi^2 \geq 0$ and $\Delta\pi =
\tfrac{1}{2}(\Delta\pi^1 + \Delta\pi^2)$.  We extend this idea
slightly, which allows us to handle the discontinuous case conveniently.

\begin{lemma}~\label{lemma:tight-implies-tight}
  Let $\pi\colon \R\to\R_+$ be a piecewise linear minimal valid function.
  Suppose $\pi = \frac12(\pi^1 + \pi^2)$, where $\pi^1$ and $\pi^2$ are valid
  functions. 
  Let $F \in \Delta\P_\B$ and let $(u,v) \in F$.  
  If $\Delta\pi_F(u,v)=0$ then $$\lim_{\substack{(x,y) \to (u,v)\\ (x,y) \in
      \relint(F)}}\Delta\pi^i(x,y) = 0 \quad\text{for $i=1,2$.}$$
\end{lemma}
\begin{proof}
  From definition of $\Delta\pi^1, \Delta\pi^2$, we see that $\Delta\pi =
  \tfrac{1}{2}(\Delta\pi^1 + \Delta\pi^2)$.  By
  Lemma~\ref{lem:minimality-of-pi1-pi2}, $\pi^1, \pi^2$
  are minimal and thus $\Delta\pi^i \geq 0$ for $i=1,2$. 
  If $\Delta\pi_F(u,v) = 0$,
  then, using Lemma~\ref{lemma:face-limits},
  $$
  0 = \Delta\pi_F(u,v)= \lim_{\substack{(x,y) \to (u,v)\\ (x,y) \in \relint(F)}} \Delta\pi(x,y) 
  =\lim_{\substack{(x,y) \to (u,v)\\ (x,y) \in \relint(F)}} \tfrac{1}{2} \bigl( \Delta\pi^1(x,y) + \Delta\pi^2(x,y)\bigr).
  $$
  Since the right hand side limit is zero and $\Delta\pi^1, \Delta\pi^2 \geq 0$, we must have that 
  $$
  0 = \lim_{\substack{(x,y) \to (u,v)\\ (x,y) \in \relint(F)}} \Delta\pi^1(x,y) 
  =\lim_{\substack{(x,y) \to (u,v)\\ (x,y) \in \relint(F)}}  \Delta\pi^2(x,y).
  $$
\end{proof}

\subsection{Continuity results}

We will need the following lemma and theorem on continuity.  Although similar
results appear in~\cite{infinite2}, we provide proofs of these facts to keep
this paper more self-contained.        
\begin{lemma}\label{lem:lipschitz}
If $\theta\colon \R \to \R$ is a subadditive function and $\limsup_{h\to 0}\lvert \frac{\theta(h)}{h}\rvert = L < \infty$, then $\theta(h)$ is Lipschitz continuous with Lipschitz constant $L$.
\end{lemma}
\begin{proof}
Fix any $\delta > 0$. Since $\limsup_{h\to 0}\lvert \frac{\theta(h)}{h}\rvert = L$,  there exists $\epsilon > 0$ such that for any $x,y \in \R$ satisfying $|x - y| < \epsilon$, $\frac{|\theta(x-y)|}{|x-y|} < L+\delta$. By subadditivity, $|\theta(x-y)| \geq |\theta(x) - \theta(y)|$ and so $\frac{|\theta(x) - \theta(y)|}{|x-y|} < L + \delta$ for all $x,y \in \R$ satisfying $|x - y| < \epsilon$. This immediately implies that for {\em all} $x,y \in \R$, $\frac{|\theta(x) - \theta(y)|}{|x-y|} < L + \delta$, by simply breaking the interval $[x,y]$ into equal subintervals of size at most $\epsilon$. Since the choice of $\delta$ was arbitrary, this shows that for every $\delta > 0$, $\frac{|\theta(x) - \theta(y)|}{|x-y|} < L + \delta$ and therefore, $\frac{|\theta(x) - \theta(y)|}{|x-y|} \leq L$. Therefore, $\theta$ is Lipschitz continuous with Lipschitz constant $L$.\end{proof}

\begin{theorem}
\label{Theorem:functionContinuous}
Let $\pi \colon \R \to \R$ be a minimal valid function and $\pi =
\frac{1}{2}(\pi^1 + \pi^2)$, where $\pi^1, \pi^2$ are valid functions. 
Suppose $\limsup_{h\to 0}\lvert \frac{\pi(h)}{h} \rvert < \infty$. 
Then this condition also holds for $\pi^1$ and $\pi^2$. This implies that
$\pi, \pi^1$ and $\pi^2$ are all Lipschitz continuous. 
\end{theorem}

\begin{proof}
By Lemma \ref{lem:minimality-of-pi1-pi2}, $\pi^1, \pi^2$ are minimal and thus
subadditive. Since we assume $\pi^1, \pi^2 \geq 0$, $\pi = \frac{1}{2}(\pi^1 +
\pi^2)$ implies that $\pi^i \leq 2\pi$ for $i = 1,2$. Therefore if
$\limsup_{h\to 0}\lvert \frac{\pi(h)}{h} \rvert = L < \infty$, then
$\limsup_{h\to 0}\lvert \frac{\pi^i(h)}{h} \rvert \leq 2L< \infty$ for
$i=1,2$. 
Applying Lemma~\ref{lem:lipschitz}, we get Lipschitz continuity for all three
functions.
\end{proof}

\subsection{Finitely generated reflection groups from additivity relations}        \label{section:Reflection-group-lemma}

We need to study the pairs $(u,v)$ where the subadditivity condition is
satisfied at equality, i.e., $\pi(u) + \pi(v) = \pi(u+ v)$, or,
equivalently, $\Delta\pi(u,v)=0$.  By
Lemma~\ref{lem:tightness}, these additivity relations also hold for
subadditive functions $\pi^1, \pi^2$ if $\pi = \frac12(\pi^1+\pi^2)$.  By
introducing the difference function (perturbation) $\bar\pi = \pi^1 - \pi$, we
can write $\pi^1 = \pi + \bar\pi$ and $\pi^2 = \pi - \bar\pi$.  Then it
follows  that the same
additivity relations also hold for~$\bar\pi$. 

The standard way to use these additivity relations is via the Interval Lemma
(Lemma~\ref{lem:interval_lemma}), where one looks for closed non-degenerate
intervals $U$, $V$, and $W=U+V$ such that
\begin{displaymath}
  \pi(u) + \pi(v) =
  \pi(u+v)\quad
  \text{for all $u \in U$ and $v\in V$ with $u+v\in W$.} 
\end{displaymath}
We will follow this standard way in section~\ref{section:AI}, where we also
extend it by a ``patching'' technique to cases where $W$ is a subset of the
Minkowski sum $U+V$.  

In this section, we develop a new way to use these additivity relations, which
complements the use of the Interval Lemma. 
Here we consider such relations when one of 
$U$, $V$, or $W$ is a single point, instead of a non-degenerate interval, and
$W$ is not necessarily the Minkowski sum of~$U$ and $V$. 
Let us assume that 
$\pi\colon\R\to\R$ satisfies finitely many classes of relations of the type
\begin{subequations} 
  \label{eq:translation-and-reflection-equations}
  \begin{equation}\label{eq:translation-equation}
    \pi(t_i) + \pi(y) = \pi(t_i + y), \quad i=1,\dots,m
  \end{equation}
  for all $y$ in some interval $Y_i$ and
  \begin{equation}\label{eq:reflection-equation}
    \pi(x) + \pi(r_i-x) = \pi(r_i),\quad i=1,\dots,n
  \end{equation}
\end{subequations}
for all $x$ in some interval $X_i$, where $t_1, \ldots, t_m$ and $r_1, \ldots,
r_n$ are finitely many points in $\R$.  

Under some conditions we will be able to construct a perturbation~$\bar\pi$
that also satisfies all equations~\eqref{eq:translation-and-reflection-equations}.
Our strategy is to first construct a function~$\psi$ that satisfies more
conditions, namely equation~\eqref{eq:translation-equation} for \emph{all} $y \in
\R$ and equation~\eqref{eq:reflection-equation} for \emph{all} $x\in \R$. 

For this construction we use methods of group theory
\cite{hall1976grouptheory}, which provide 
fundamental insights into the structure of the perturbations.  Readers who are
unfamiliar with the terminology of group theory may skip
the following development and verify Lemma~\ref{lemma:our-equivariant-psi} below
by elementary means. This will be sufficient for following the proofs of the main
theorems of this paper, which are proved in section~\ref{section:proofs-main}.
However, the construction of the extreme function with irrational breakpoints
in section~\ref{sec:irrational} requires the group theoretic tools of this section.

We consider a subgroup of the group $\Aff(\R)$ of invertible affine linear transformations of~$\R$ as follows.
\begin{definition}
  For a point $r \in \R,$ define the \emph{reflection} $\rho_r\colon \R\to \R$, $x
  \mapsto r-x$.  For a vector $t \in \R,$ define the \emph{translation}
  $\tau_t\colon \R\to \R$, $x \mapsto x + t$.  
\end{definition}

\begin{figure}[t!]
\centering
(a) \ifpdf
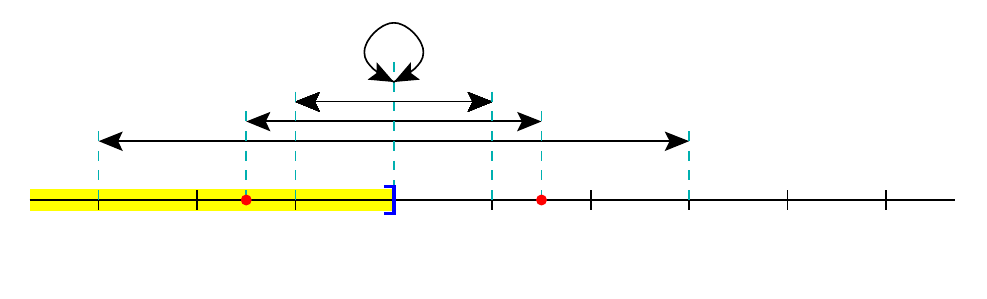
\else
\input{figures/reflectionGroup1d.pstex_t}
\fi\par
(b) \ifpdf
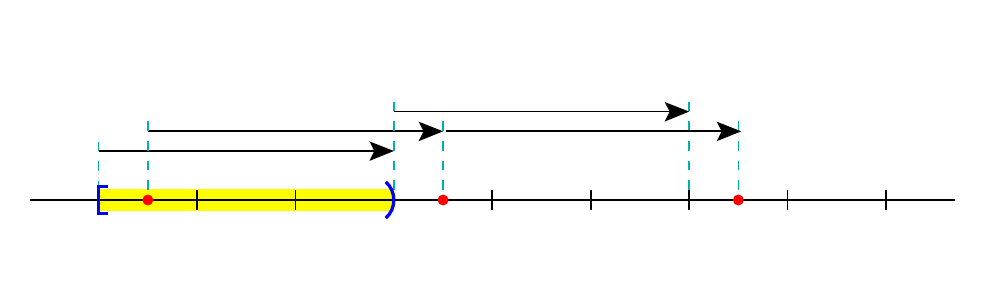
\else
\input{figures/reflectionGroup1d-trans.pstex_t}
\fi\par
(c) \ifpdf
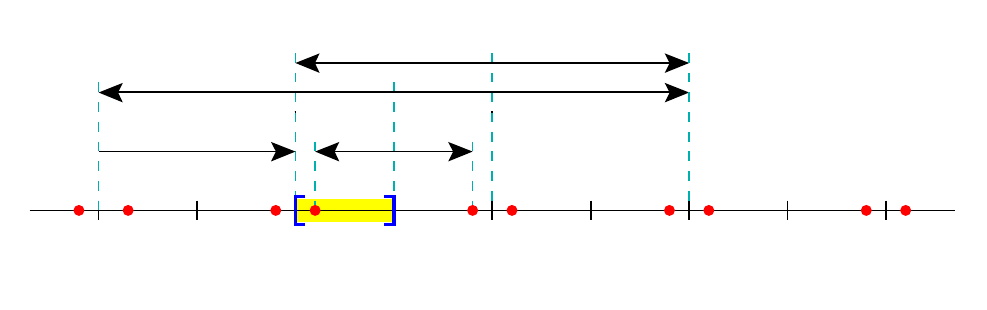
\else
\input{figures/reflectionGroup1d-2refs.pstex_t}
\fi
\caption{Finitely generated reflection groups and their major features.
  Forward arrows indicate translations~$\tau_t$ by a number~$t$, bidirected arrows indicate
  point reflections~$\rho_r$ about the invariant point~$\frac{r}2$.  
  The red filled circles mark points that form an
  orbit $\Gamma(x)$ of a typical point~$x$ under the reflection group~$\Gamma$. 
  The blue-yellow interval is a fundamental domain~$V^+$ of
  the reflection group, i.e., a system of representatives of the orbits.
  (a) The group~$\Gamma$ generated by a single point reflection~$\rho_{3/4}$.
  The fundamental domain is an unbounded closed interval.
  (b) The group~$\Gamma$ generated by a single translation~$\tau_{3/8}$.  The
  fundamental domain is a half-open interval.
  (c) The group~$\Gamma$ generated by two reflections $\rho_{r_1} =
  \rho_{3/4}$ and $\rho_{r_2}=\rho_{1}$.  Its normal
  subgroup~$\Tau$ of translations is generated by the
  translation~$\tau_t=\tau_{r_2-r_1}=\rho_{r_2}\circ\rho_{r_1} = \tau_{1/4}$. 
  The fundamental domain given by Lemma~\ref{lemma:phi} 
  is the closed interval $[\frac12 r_1, \frac12(r_1+t)]=[\frac14,\frac38]$. 
}\label{fig:reflection-groups}
\end{figure}

Given a finite number of points $r_1,\dots,r_n$ and a finite number of
vectors~$t_1,\dots,t_m$, we will define the subgroup $$\Gamma= \langle
\rho_{r_1},\dots,\rho_{r_n}, \tau_{t_1},\dots,\tau_{t_m}\rangle$$  
that is generated by the listed translations and reflections.
Figure~\ref{fig:reflection-groups} shows examples of such finitely generated
reflection groups.

Let $r,s,w,t\in \R$.  Each reflection is an involution: $\rho_r \circ \rho_r =\textit{id}$,  two reflections give one translation: $\rho_r \circ \rho_s =
\tau_{r-s}$.  Thus, if we assign a \emph{character} $\chi(\rho_r) = -1$ to
every reflection and $\chi(\tau_t) = +1$ to every translation, then this
extends to a \emph{group character} of~$\Gamma$, that is, a group homomorphism
$\chi\colon \Gamma\to\{\pm1\}\subset\C^\times$.  

On the other hand, not all pairs of reflections need to be
considered: $\rho_{s}\circ \rho_{w} = (\rho_{s}\circ \rho_r) \circ (\rho_r
\circ \rho_{w}) = (\rho_{r}\circ \rho_s)^{-1} \circ (\rho_r \circ \rho_{w})$.
Thus the subgroup~$\Tau = \ker(\chi) = \{\, \gamma \in \Gamma \st \chi(\gamma) = +1\,\}$ 
of translations in~$\Gamma$ is generated as
follows.  $$\Tau = 
\langle \tau_{r_2-r_1}, \dots, \tau_{r_n-r_1}, \tau_{t_1},\dots,\tau_{t_m}
\rangle.$$  It is \emph{normal} in~$\Gamma$, as it is stable by conjugation by any
reflection: $\rho_{r} \circ \tau_t \circ \rho_r^{-1} = \tau_{-t}$.  

In the following, we assume that $n\geq1$, i.e., at least one of the
generators is a reflection.
Now, if $\gamma \in \Gamma$ is
not a translation, i.e., $\chi(\gamma) = -1$, then it is generated by an odd
number of reflections, and 
thus can be written as $\gamma = \tau \circ \rho_{r_1}$ with $\tau\in \Tau$.  Thus $\Gamma /
\Tau \cong \langle\rho_{r_1}\rangle$ is of order~$2$. In short, we have the following
lemma. 
\begin{lemma}  \label{lemma:semidirect}
  Let $n\geq1$.  Then the group $\Gamma=\langle
  \rho_{r_1},\dots,\rho_{r_n}, \tau_{t_1},\dots,\tau_{t_m}\rangle$ is the
  semidirect product $\Tau \rtimes \langle \rho_{r_1} \rangle$,  
  where the (normal) subgroup of translations is of index~$2$ in~$\Gamma$ and
  can be written as $$\Tau = \{\, \tau_t \st t \in 
  \Lambda\,\},$$ where $\Lambda$ is the additive subgroup of~$\R$ generated
  by $r_2-r_1, \dots, r_n-r_1,
  {t_1},\dots,{t_m}$, 
  $$\Lambda = {\langle r_2-r_1, \dots, r_n-r_1,
  {t_1},\dots,{t_m} \rangle}_\Z \subseteq \R.$$%
\end{lemma}%
The additive subgroup~$\Lambda$ distinguishes two qualitatively different
types of finitely generated reflection groups.  If the additive
subgroup~$\Lambda$ is discrete, we are able to construct continuous functions
that are $\Gamma$-equivariant, i.e., invariant under all translations
in~$\Tau$ and odd w.r.t.\ all reflections $\rho_{r_1},\dots,\rho_{r_n}$. We
will use these functions in section~\ref{sec:non-extreme-by-perturbation}.
On the other hand, if $\Lambda$ is not discrete, we can show that there is no
nontrivial continuous $\Gamma$-equivariant function; we prove a stronger version
of this statement in section~\ref{sec:irrational} and use it to show the extremality of
a certain minimal function. 

\begin{figure}[t!]
  \centering
  \ifpdf
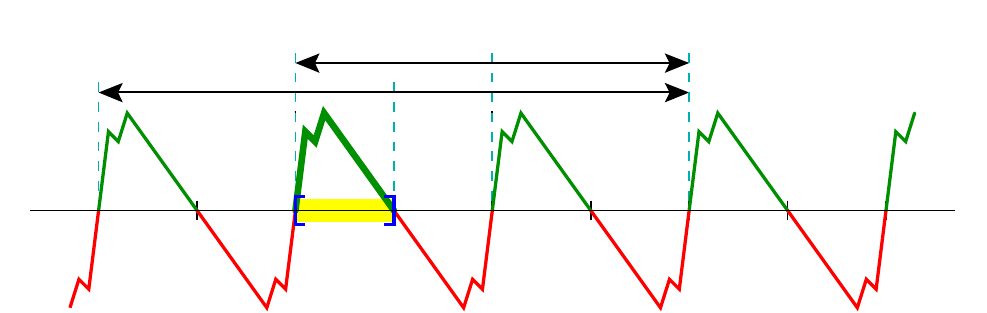
\else
\input{figures/reflectionGroup1d-2refs-equi.pstex_t}
\fi
  \caption{A~$\Gamma$-equivariant function~$\psi$ for the reflection group in Figure~\ref{fig:reflection-groups}\,(c),
    constructed according to Lemma~\ref{lemma:phi} by choosing an arbitrary
    function on the fundamental domain~$V^+$ that is zero on the boundary
    (endpoints) of~$V^+$ (heavy green lines) and extending to all of~$\R$
    (lighter green and red lines) equivariantly.}
  \label{fig:equivariant-general}
\end{figure}

\begin{definition}
 The \emph{orbit} of a point~$x\in\R$ under the group~$\Gamma$ is the set
 \begin{displaymath}
   \Gamma(x) = \{\, \gamma(x) \st \gamma\in\Gamma \,\}.
 \end{displaymath}
\end{definition}
Figure~\ref{fig:reflection-groups} shows orbits~$\Gamma(x)$ of typical points~$x$.
\begin{definition}
 A \emph{fundamental domain} of~$\Gamma$ is a subset of~$\R$ that is a
    system of representatives of the orbits. 
\end{definition}

If $\Lambda$ is discrete, then since $\Lambda\subset\R$, 
we know that $\Lambda$ is generated by one number~$t$. (We
denote this fact by writing $\Lambda = \langle t\rangle_\Z$.) 
We then find a nondegenerate
interval that serves as a fundamental domain of~$\Gamma$; see again
Figure~\ref{fig:reflection-groups} for examples.  We then construct continuous
$\Gamma$-equivariant functions as follows. 
\begin{lemma}[\unboldmath Construction of $\Gamma$-equivariant functions]
  \label{lemma:phi}
  Suppose that the additive subgroup $\Lambda$ defined above is discrete, and
  let $t \in \Lambda$ such that $\Lambda = \langle t\rangle_\Z$.
  
  Then $V^+ = [\frac12(r_1-t),\frac12 r_1]$ is a fundamental domain of~$\Gamma$.
  Let $\psi\colon V^+\to \R$ be any function such that
  $\psi\big|_{\partial V^+} = 0$,
  where $\partial V^+$ denotes the boundary of~$V^+$.
  Then the \emph{equivariance formula} 
  \begin{equation}\label{eq:equivariance}
    \psi(\gamma(x)) = \chi(\gamma) \psi(x)
    \quad\text{for $x\in \R$ and $\gamma \in \Gamma$}
  \end{equation}
  gives a well-defined extension of $\psi$ to all of~$\R$.
  It satisfies equations~\eqref{eq:translation-and-reflection-equations} for
  all $x,y\in \R$.
\end{lemma}

Figure~\ref{fig:equivariant-general} illustrates this construction.

\begin{proof}
  The reflection~$\rho_{r_1}$ sends $V^+$ to $V^-=[\frac12 r_1,
  \frac12(r_1+t)]$.  The sets $V^+$ and $V^-$ intersect in one point~$\frac12
  r_1$, which is invariant under~$\rho_{r_1}$. 
  Then the translations $\tau\in\Tau$ tile all of~$\R$ with copies of $V^+$ and
  $V^-$.  

  Thus $V^+$ is a fundamental domain for~$\Gamma$.
  Since on the boundary, $\psi=0$, the extension is well-defined. Since $\psi(t_i) = 0$ and
  $\psi(r_i) = 0$, equation~\eqref{eq:equivariance}
  implies~\eqref{eq:translation-and-reflection-equations}. 
\end{proof}

\begin{remark}
  The same construction works if we consider reflection groups of~$\R^k$, when
  $\Lambda$ is a lattice of~$\R^k$.   The fundamental
  domain $V^+$ can be chosen as one half of a Voronoi cell or one half of the
  fundamental parallelepiped of the lattice~$\Lambda$.  This will become important in
  \cite{basu-hildebrand-koeppe:equivariant-2,bhk-IPCOext}. 
\end{remark}

Of particular interest to us is the case of the
reflection group~$\Gamma = \langle \rho_{1/q}, \tau_{1/q} \rangle$, where the
integer $q$ is the least common multiple of all denominators 
of the (rational) breakpoints of a piecewise linear function~$\pi$.  Then
$\Lambda=\tfrac1q\Z$, and 
$\Gamma$ contains all reflections 
$\rho_g$ and translations~$\tau_g$ for $g \in \tfrac{1}{q} \Z$ and thus all reflections and
translations corresponding to all breakpoints of~$\pi$.  

Let the function
$\psi\colon [0,\tfrac{1}{2q}] \to \R$ be given by 
connecting the points  $(0,0), (\tfrac{1}{4q}, 1), (\tfrac{1}{2q},0)$, and
then extending $\psi$ to all of~$\R$ using Lemma~\ref{lemma:phi}, 
where $V^+ = [0,\tfrac{1}{2q}]$.
This gives the points $(\tfrac{3}{4q},-1)$ and $(1,0)$.  The function is
periodic with period~$1$. See Figure \ref{figure:phi}.   
\begin{lemma}\label{lemma:our-equivariant-psi}
  The function $\psi\colon\R\to\R$ constructed above has the following properties:
  \begin{enumerate}[\rm(i)]
  \item $\psi(g) = 0$ for all $g\in \tfrac{1}{q} \Z$,
  \item $\psi(x)= - \psi(\rho_{g}(x)) = - \psi(g - x)$ for all $g \in \tfrac{1}{q} \Z, x \in
    [0,1]$,
  \item $\psi(x) = \psi(\tau_g(x)) = \psi(g + x)$ for all $g \in \tfrac{1}{q} \Z, x \in
    [0,1]$,
  \item $\psi$ is piecewise linear with breakpoints in $\tfrac{1}{4q}\Z$.
  \end{enumerate}
\end{lemma}
\begin{proof}
  The properties follow directly from the equivariance formula~\eqref{eq:equivariance}.
\end{proof}

\begin{figure}[t!]
  \centering
  \ifpdf
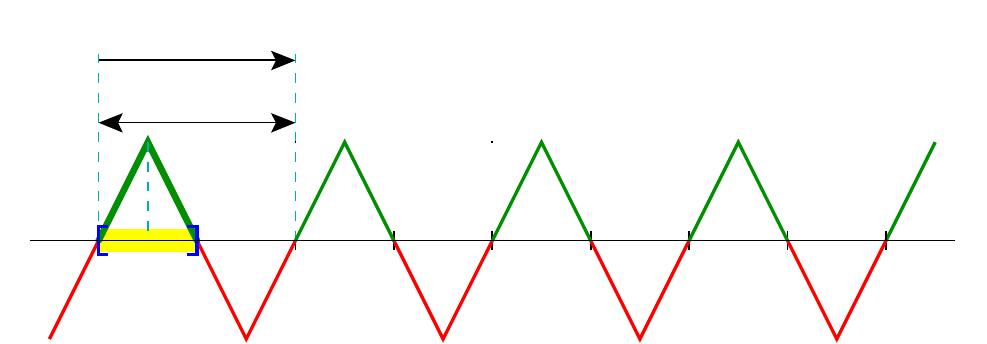
\else
\input{figures/reflectionGroup1d-1q-equi.pstex_t}
\fi
  \caption{A $\Gamma$-equivariant perturbation for the reflection
    group~$\Gamma=\langle\rho_{1/q}, \tau_{1/q} \rangle$.  The continuous piecewise
    linear function $\psi$ is defined on the fundamental domain
    $V^+=[0,\tfrac1{2q}]$ (blue-yellow interval) by connecting the points $(0,0),
    (\tfrac{1}{4q},1), (\tfrac{1}{2q},0)$ (heavy green lines) and then extended
    according to Lemma~\ref{lemma:phi} to a $\Gamma$-equivariant function
    (lighter green and red lines).}  
 \label{figure:phi}
  \end{figure} 

In section~\ref{sec:non-extreme-by-perturbation} we will modify the
function~$\psi$ to become a suitable perturbation~$\bar\pi$.

\section{A finite system of linear equations} 
\label{section:system}

A technique for investigating whether a minimal valid function~$\pi$ is
extreme is to test whether a certain finite-dimensional system of linear equations
has a unique solution.  
We construct the system in section~\ref{section:def}.
If the system does not have unique solution, we can always construct perturbations
of~$\pi$ that prove that $\pi$ is not extreme
(section~\ref{section:system-necessary}
).  However, the other direction
does not hold in general.  A main contribution of the present paper is to find
the precise conditions under which it does hold.  
We take the first step in section~\ref{section:system-AI}, which prepares the
complete solution in section~\ref{section:proofs-main}.

\subsection{Definition of the system}
\label{section:def} 

In this subsection we suppose that $\pi$ is a minimal valid function 
that is piecewise linear with breakpoints in~$B$. 
We will set up a system that tests whether there exist distinct minimal functions
$\pi^1,\pi^2$ such that $\pi = \frac12(\pi^1+\pi^2)$
that are piecewise linear functions with breakpoints
in $\B$.  

Suppose $\pi^1$ and $\pi^2$ are minimal functions such that  $\pi =
\frac12(\pi^1+\pi^2)$ that are piecewise linear functions with breakpoints
in $\B$. 
For a face $F\in\Delta\P_\B$, define $$\Delta\pi^i_F(x,y) = \pi^i_I(x)
+\pi^i_J(y) - \pi^i_K(x+ y),$$ where $F = F(I,J,K) = \{\,(x,y) \st x \in I,\, y \in
J,\, x+ y \in K\,\}$ and $I,J,K \in \I_{\B,\edge} \cup \I_{\B,\point}$.  
Then the following is a direct corollary of
Lemma~\ref{lemma:tight-implies-tight}. 

\begin{corollary}~\label{cor:tight-implies-tight-with-pi1F}
Let $F \in \Delta\P_\B$ and let $(u,v) \in F$.  If
$\Delta\pi_F(u,v)=0$ then $\Delta\pi_F^i(u,v) = 0$ for $i=1,2$. 
\end{corollary}

We now set up a system of finitely many linear equations in finitely many
variables that $\pi$ satisfies and that
$\pi^1$ and $\pi^2$ must also satisfy under the assumption that they are
piecewise linear functions with breakpoints 
in $\B$.

Let $\varphi$ be an arbitrary piecewise linear function with breakpoints in
$\B$ that is periodic modulo~$\Z$.  
Then, by definition, $\varphi(x) = \varphi_I(x)$ for $x \in \relint(I)$, where 
$\varphi_I(x) = m_I x + b_I$ for $I \in \I_{\B,\edge}$ and $x\in\R$ and
$\varphi_I(x) = b_I$ for all $I \in \I_{\B,\point}$ and $x\in\R$.  
For every $F\in \Delta\P_\B$, let $\Delta\varphi_F(x,y) = \varphi_I(x)
+\varphi_J(y) - \varphi_K(x+ y)$ for $x,y\in\R$, where $F = F(I,J,K) = \{\,(x,y) \st x \in I, y \in J, x+ y \in K\,\}$ and $I,J,K \in \I_{\B,\edge} \cup \I_{\B,\point}$.  
Consider such functions~$\varphi$ that satisfy the following system of linear
equations in terms of $m_I, b_I$ for $I \in \I_{\B,\edge}$ with $I\subseteq[0,1]$ and $b_I$ for $I
\in \I_{\B,\point}$ with $I\subseteq[0,1]$: 
\begin{equation}
\label{equation:system}
\begin{cases}
\varphi(0) = 0,\\
\varphi(f) = 1,\\
\varphi(1) = 0,\\
\begin{aligned}
  \Delta\varphi_F(u,v) = 0 
  \quad
  & \text{for all $(u,v) \in \verts(F)$ with $\Delta\pi_F(u,v) = 0$} \\
  & \text{where $F \in \Delta\P_\B$, $F\subseteq[0,1]^2$.}
  \end{aligned}
\end{cases}
\end{equation}

 Since $\varphi=\pi$ satisfies the system of equations, we know that the system has a solution.

\subsection{Necessary condition for extremality}
\label{section:system-necessary}

We now prove the following theorem.
\begin{theorem}
If $\pi$ is a piecewise linear valid function with breakpoints in $\B$ and  the system of equations \eqref{equation:system} does not have a unique solution, then $\pi$ is not extreme.
\label{theorem:systemNotUnique}
\end{theorem}
\begin{proof}
  If $\pi$ is not minimal, then $\pi$ is not extreme.  Thus, in the following,
  we assume that $\pi$ is minimal. 
Suppose \eqref{equation:system} does not have a unique solution. Let 
\begin{displaymath}
  \{(\bar
m_I, \bar b_I)_{I\in \I_{\B,\edge}, I\subseteq[0,1]},\allowbreak (\bar b_I)_{I
  \in \I_{\B,\point}, I\subseteq[0,1]}\}
\end{displaymath}
be a non-trivial element in the kernel of the system above.  Let $\bar
\varphi$ be the piecewise linear function periodic modulo~$\Z$, given by
\begin{displaymath}
\bar
\varphi(x) =
\begin{cases}
  \bar m_I x + \bar b_I
  & \text{for $x \in \intr(I)$ where $I \in \I_{\B,\edge}$,
    $I\subseteq[0,1]$,}\\
  \bar b_I 
  & \text{for all $\{x\} = I \in
    \I_{\B,\point}$, $I\subseteq[0,1].$}
\end{cases}
\end{displaymath}
(The periodic extension from $[0,1]$ to $\R$ is well-defined because
$\bar\varphi(0) = \bar\varphi(1)$.) 

Then for any $\epsilon$, $\pi + \epsilon \bar\varphi$ also satisfies the system of equations. 
Let
$$
\epsilon = \min \{\,\Delta\pi_F(x,y)\neq 0 \st F \in \Delta\P_\B, \
F\subseteq[0,1]^2, \ (x,y) \in \verts(F)\,\},
$$ 
which is well-defined and positive as a minimum over a finite number of
positive values,
and set 
\begin{displaymath}
  \pi^1 = \pi + \frac{\epsilon}{3 ||\bar\varphi||_\infty} \bar\varphi,
  \quad
  \pi^2 = \pi - \frac{\epsilon}{3 ||\bar\varphi||_\infty} \bar\varphi.
\end{displaymath}
Note that $0 < ||\bar\varphi||_\infty < \infty$ since
$\bar\varphi$ comes from a non-trivial element in the kernel, and because it
is piecewise linear on a compact domain.  We claim that $\pi^1, \pi^2$ are
both minimal.  We show this for $\pi^1$, and $\pi^2$ is similar.  
Since $\pi$ satisfies the system~\eqref{equation:system} and $\bar\varphi$ is an
element of the kernel, $\pi^1$ satisfies the system~\eqref{equation:system} as
well. In particular, we have $\pi^1(0) = 0, \pi^1(f)= 1, \pi^1(1) = 0$. 

Next, we show that $\pi^1$ satisfies the symmetry test of
Theorem~\ref{minimality-check}.   
To this end, first note that $\varphi(f) = 1$ is an equation
in~\eqref{equation:system}.
Also, since $\pi$ is minimal, $\Delta\pi_F \equiv 0$ whenever $F \subseteq
\{\,(x,y)\in[0,1]^2 \st x+ y \equiv 
f\pmod1\,\}$.  Therefore $\Delta\varphi_F(u,v) = 0$ is an equation in
\eqref{equation:system} for each $(u,v) \in \verts(F)$.  

Lastly, we show that $\pi^1$ satisfies the subadditivity test of Theorem~\ref{minimality-check}.  Let $F \in \Delta\P_\B$, $F\subseteq[0,1]^2$, and $(u,v) \in \verts(F)$.  If $\Delta\pi_F(u,v) = 0$, then $\Delta\varphi_F(u,v) = 0$, as implied by the system of equations.  Otherwise, if $\Delta\pi_F(u,v) > 0$, then   
\begin{align*}
  \Delta\pi_F^1(u,v) 
  &= \Delta\pi_F(u,v) + \frac{\epsilon}{3
    ||\bar\varphi||_\infty} \bar\varphi(u) + \frac{\epsilon}{3
    ||\bar\varphi||_\infty} \bar\varphi(v) - \frac{\epsilon}{3
    ||\bar\varphi||_\infty} \bar\varphi(u+ v) \\
  &\geq \Delta\pi_F(u,v) -
  \tfrac{\epsilon}{3 } - \tfrac{\epsilon}{3} - \tfrac{\epsilon}{3 } \geq 0
\end{align*}

Therefore, by Theorem~\ref{minimality-check}, $\pi^1$ (and, by the same
argument, $\pi^2$) is a minimal valid function.  Therefore $\pi$ is not
extreme. 
\end{proof}

\subsection{Sufficient condition for extremality of an affine imposing function~$\pi$}
\label{section:system-AI}


Consider the following
definition.

\begin{definition}
  Let $\pi$ be a minimal valid function.  
  \begin{enumerate}[\rm(a)]
  \item For any closed proper interval $I \subset [0,1]$, if $\pi$ is affine in
    $\intr(I)$ and if for all valid functions $\pi^1, \pi^2$ such that $\pi =
    \tfrac{1}{2}(\pi^1 + \pi^2)$ we have that $\pi^1, \pi^2$ are
    affine in $\intr(I)$, then we say that $\pi$ is \emph{affine imposing in
      $I$}.
  \item For a collection $\I$ of closed proper intervals of $[0,1]$, if for all
    $I \in \I$, $\pi$ is affine imposing in $I$, then we say that $\pi$ is
    \emph{affine imposing in $\I$}.
  \end{enumerate}
\end{definition}

\begin{corollary}
\label{corollary:iff}
If $\pi$ is a minimal piecewise linear function with breakpoints in $\B$ and
is affine imposing in $\I_{B,\edge}$, then $\pi$ is extreme if and only if the
system of equations \eqref{equation:system} has a unique solution.   
\end{corollary}

Later, in section~\ref{section:proofs-main}, we will determine when $\pi$ has
this property.

\begin{proof}
Suppose there exist distinct, valid functions $\pi^1, \pi^2$ such that $\pi =
\tfrac{1}{2}(\pi^1 + \pi^2)$. 

Since $\pi$ is affine
imposing in $\I_{\B,\edge}$, $\pi^1$ and $\pi^2$ 
must also be piecewise linear functions with breakpoints in the same set~$\B$.
Also, since $\pi$ is minimal, $\pi^1$ and $\pi^2$ are both minimal by
Lemma~\ref{lem:minimality-of-pi1-pi2}.

Furthermore, $\pi$ and, by
Lemma~\ref{lemma:tight-implies-tight}, also $\pi^1, \pi^2$ satisfy the system
of equations \eqref{equation:system}.  If this system has a unique solution,
then $\pi = \pi^1  = \pi^2$, which is a contradiction since $\pi^1, \pi^2$
were assumed distinct.  Therefore $\pi$ is extreme. 

On the other hand, if the system \eqref{equation:system} does not have a unique solution, then by Theorem~\ref{theorem:systemNotUnique}, $\pi$\ is not extreme.
\end{proof}

\subsection{Remarks on reducing the dimension of the system}

Writing down a reduced system is advantageous for reading a proof of a
function being extreme. In previous literature, this has been done in two main ways.  

\begin{remark}
  If $\pi$ is continuous, the variables $b_I$ for any $I \in
  \I_{\B,\point}$ become redundant and can be removed from the system.  Also,
  it follows that any solution~$\varphi$ to the system (and thus functions
  $\pi^1, \pi^2$) also must be continuous, so we can remove the variables
  $b_I$ for all $I \in \I_{\B,\edge}$ and replace these values as integrals
  over the function, which are linear in the slopes $c_I$ for $I \in
  \I_{\B,\edge}$.
\end{remark}
\begin{remark}
As we will show in section~\ref{section:proofs-main}, 
each of the functions $\pi$, $\pi^1$, $\pi^2$ 
actually has the same slopes on certain
intervals.
Therefore, the number of variables can be reduced.  This
observation was first applied to prove the two slope theorem \cite{tspace}.
\end{remark}

\section{The rational case: Proof of the main results}
\label{section:proofs-main}

In this section, let $\pi$ be a fixed minimal valid function, whose breakpoints are all rational.  Let
$q$ be the least common multiple of all denominators of breakpoints.
Then the additive group $\langle \B \rangle_\Z$ generated by the set~$B$ of
breakpoints is $\tfrac{1}{q} \Z$.
We will think of $\pi$ as a piecewise linear function with
breakpoints in $\tfrac{1}{q} \Z$.  Consequently we will consider the one-dimensional polyhedral
complex $\P_{\frac1q\Z}$ instead of~$\P_B$.  We will use the abbreviation
$\P_q = \P_{\frac1q\Z}$; so $\P_q = \{\emptyset\}\cup \Ipoint \cup \Iedge$, etc.  

Likewise we will consider the two-dimensional polyhedral
complex~$\Delta\P_{q} = \Delta\P_{\frac1q\Z}$ introduced in section \ref{section:minimalityTest}.
Observe that because of the even spacing of the set of breakpoints, 
every face of~$\Delta\P_q$ is a simplex (i.e., a point, edge,
or triangle).  See Figure
\ref{figure:polyhedralComplex}. 

Later we will also refine $\tfrac{1}{q} \Z$ to $\tfrac{1}{4q} \Z$ 
and use the
corresponding one-dimensional polyhedral complex $\P_{4q}$ and two-dimensional
polyhedral complex $\Delta\P_{4q}$.  

In the following, we show that either $\pi$ is affine imposing in $\Iedge$
(section~\ref{section:AI}) or we can construct a piecewise linear
$\Gamma$-equivariant perturbation with breakpoints in $\frac1{4q}\Z$ that proves
$\pi$ is not extreme (section~\ref{sec:non-extreme-by-perturbation}). 
If $\pi$ is affine imposing in $\Iedge$, we use Corollary~\ref{corollary:iff} 
to decide if~$\pi$ is extreme or not
(section~\ref{section:system-use}).   
In section~\ref{sec:proof-intro-thm}, we prove the main theorems stated in
the introduction.

\subsection{Imposing affine linearity on open intervals}
\label{section:AI}

In this subsection we find a set $\S^2_{q,\edge}$ of intervals in which the
function~$\pi$ is affine imposing.

\paragraph{Covered intervals.}
In the first step, we consider certain projections of the $2$-faces (triangles)
$F$ of the complex~$\Delta\P_q$ with $\Delta\pi=0$ on 
$\intr(F)$.  (These triangles are shaded bright green in
Figure~\ref{figure:polyhedralComplex}.)
We define the projections $p_1,p_2,p_3\colon \R\times \R \to \R$ as 
\begin{displaymath}
  p_1(x,y) = x,\quad p_2(x,y) = y, \quad\text{and}\quad p_3(x,y) = x+ y.
\end{displaymath}
\begin{figure}[t!]%
\centering%
\ifpdf
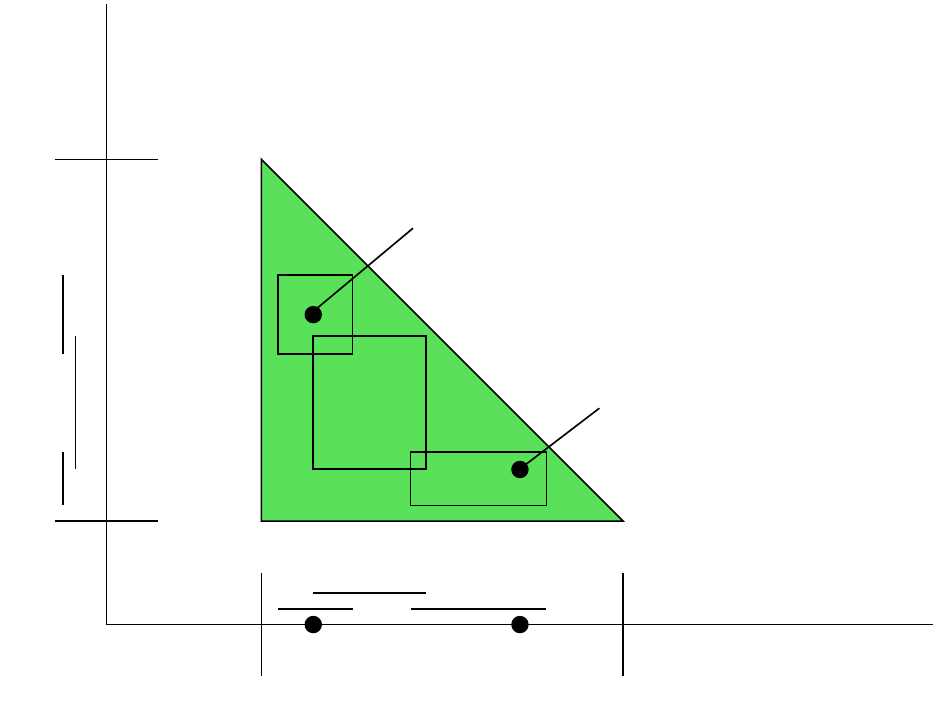
\else
\input{figures/figurePatching1D.pstex_t}
\fi%
\caption{The construction of a chain of patches in the proof of Lemma~\ref{lemma:patching-interval-lemma}}%
\label{figure:patching1D}%
\end{figure}%
Let $F$ be one of these faces.  We apply the Interval Lemma
(Lemma~\ref{lem:interval_lemma})  
to intervals $U$, $V$, $U+V$ such that the two-dimensional ``patch'' $U\times
V$ lies entirely in the face $F$.  By covering the
interior of $F$ with such patches, we show below that $\pi$ is affine
imposing in the projections $p_1(F)$, $p_2(F)$, and $p_3(F)$.  This is a standard technique in the literature, which we make explicit here
as a lemma.
We note that
$p_1(F)$, $p_2(F)$, and $p_3(F)$ are $1$-faces (intervals) of~$\P_q$, and
$\intr(p_i(F)) = p_i(\intr(F))$ for $i=1,2,3$. 
\begin{lemma}
  \label{lemma:patching-interval-lemma}
  Let $F\in\Delta\P_q$ be a two-dimensional face with $\Delta\pi=0$ on $\intr(F)$. 
  \begin{enumerate}[\rm(i)]
  \item Then $\pi$ is affine imposing in the intervals~$p_1(F)$, $p_2(F)$, and $p_3(F)$.
  \item More specifically, let $\pi^1, \pi^2$ be valid functions such that
    $\pi = \tfrac{1}{2} (\pi^1 + \pi^2)$.  Let $\theta=\pi$, $\pi^1,$
    or $\pi^2$. 
    Then the function~$\theta$ is affine with the same slope in $\intr(p_1(F))$,
    $\intr(p_2(F))$, and $\intr(p_3(F))$.     
  \end{enumerate}
\end{lemma}
We will say that the intervals $p_1(F)$, $p_2(F)$, and $p_3(F)$ in
Lemma~\ref{lemma:patching-interval-lemma} are \emph{covered}.
\begin{proof}
Let $\pi^1,\pi^2$ be valid functions such that $\pi = \frac12 (\pi^1+
\pi^2)$.
We will first show that $\pi^1$ is affine on the interior of the interval $I = p_1(F)$.  

Fix any $x_0 \in \intr(I)$. We will show that there exists $c \in \R$ such that $\pi^1(x_*) = \pi^1(x_0) + c\cdot (x_*-x_0)$ for all $x_* \in \intr(I)$. This will prove the claim.

Let $x_* \in \intr(I)$.  Since $x_0,x_* \in \intr(I)$, there exist points $(x_0,
y_0), (x_*, y_*) \in \relint(F)$ such that $x_0 = p_1 (x_0, y_0)$
and $x_* = p_1 (x_*, y_*)$. 
 Therefore, we can construct a sequence of closed intervals $U_0, \ldots, U_n
 \subseteq \intr(I)$ and another sequence of closed intervals $V_0, \ldots, V_n$ such that $U_i \times V_i \subseteq \relint(F)$ and $(x_0, y_0) \in U_0\times V_0$ and $(x_*, y_*) \in U_n \times V_n$ such that $\intr(U_{i-1} \times V_{i-1}) \cap \intr(U_{i} \times V_{i}) \neq \emptyset$, for all $i = 1, \ldots, n$. 
(See Figure~\ref{figure:patching1D}.)
Therefore, we can find a sequence of points $(x_1, y_1), \ldots, (x_n, y_n)$
such that $(x_i, y_i) \in\intr(U_{i-1} \times V_{i-1}) \cap \intr(U_{i} \times
V_{i})$. 

Now, since $\Delta\pi(x,y) = 0$ over $\relint(F)$, we have that
$\Delta\pi(u,v) = 0$ for all $(u,v) \in U_i\times V_i$, $i = 0, \ldots,
n$. This implies $\pi(u) + \pi(v) = \pi(u+ v)$ for all $(u,v) \in
U_i\times V_i$, and so the same relation holds for $\pi^1$ by
Lemma~\ref{lem:tightness}. Using the Interval Lemma (Lemma~\ref{lem:interval_lemma}), there exists $c_i$ such
that $\pi^1(u) = \pi^1(x_i) + c_i\cdot(u - x_i)$, for all $u \in U_i$ and this
holds for every $i = 0, \ldots, n$. Observe that $x_i$ belongs to the interior
of both $U_{i-1}$ and $U_i$ for all $i= 1, \ldots, n$. Therefore, we must have
$c_{i-1} = c_i$ for all $i = 1, \ldots, n$; we take $c$ to be this common
value. Now using the relation $\pi^1(x_*) = \pi^1(x_0) + (\sum_{i=1}^n
\pi^1(x_i) - \pi^1(x_{i-1})) + (\pi^1(x_*) - \pi^1(x_n))$, we find that
$\pi^1(x_*) = \pi^1(x_0) + c\cdot(x_* - x_0)$.\smallbreak 

Because of the symmetry of~$\Delta\pi$ in its arguments $x$ and $y$, 
the case $I = p_2(F)$ does not need a separate proof. 

Applying the Interval Lemma (Lemma~\ref{lem:interval_lemma}) with $U_0$ and
$V_0$, we obtain that $\pi^1$ has the same slope on $\intr(p_1(F))$ and
$\intr(p_2(F))$.  Since $\pi^1(x) + \pi^1(y) = \pi^1(x+ y)$ holds for
$(x,y)\in F$, it follows that $\pi^1$ is affine in $p_3(F)$ with the same slope. 

The same argument can be made for~$\pi^2$.   Thus the claim also holds for $\pi =
\frac12(\pi^1+\pi^2)$. 
\end{proof}

We define the set of covered intervals,
$$
  \I^2_{q,\edge} = \biggl\{\, I\in \Iedge \mathrel{\bigg|}
  \begin{array}{@{}l@{}}
    I = p_1(F) \text{ or } I = p_2(F) \text{ or } I = p_3(F) \\
    \text{for some $2$-face $F \in \Delta\P_q$ with $\Delta\pi = 0$ on $\intr(F)$}
  \end{array}
  \,\biggr\}.
$$
The superscript 2 indicates that this set comes from projections of
two-dimensional faces $F$.  Then we have the following corollary.
\begin{corollary}
  \label{lemma:affine-imposing-by-patching}
    $\pi$ is affine imposing in $\I^2_{q,\edge}$.
\end{corollary}%

In the example in Figure~\ref{figure:polyhedralComplex}, all intervals are
covered, so $\I^2_{q,\edge} = \Iedge$.  

\paragraph{A graph of intervals.}
Next we will define a finite graph~$\G$ whose nodes correspond to the
intervals~$I$ in 
$\Iedge$.  If the function~$\pi$ is affine-imposing in a interval~$I$, then
this property will propagate along the edges of the graph to the entire
connected component of~$I$.

\begin{figure}[t!]%
\centering
(a)\scalebox{.9}{\inputfigeps{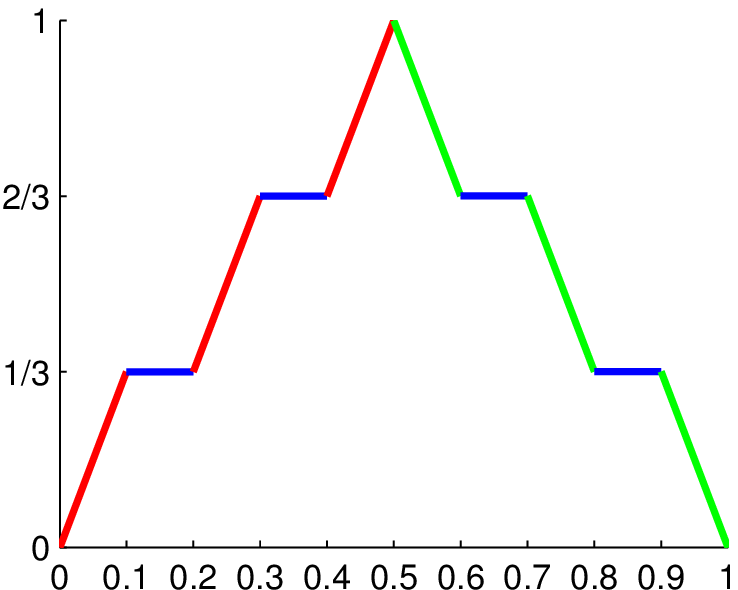}}\hfill%
(b)\scalebox{.9}{\inputfigeps{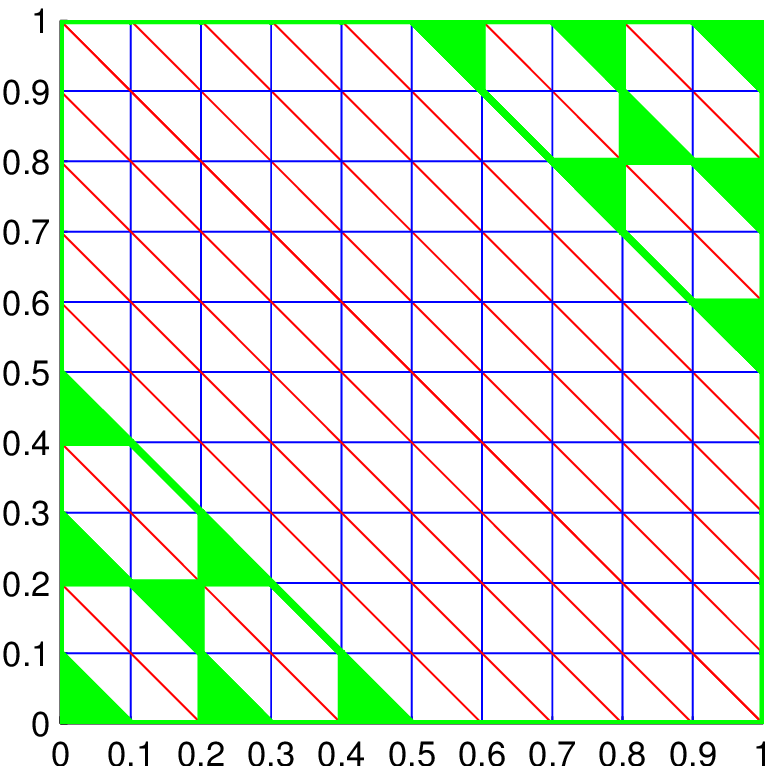}}
\caption{A minimal continuous piecewise linear function~$\pi$ and its
  polyhedral complex~$\Delta\P_q$.%
  \quad(a) The function $\pi$ with $f = 0.5$ and breakpoints in
  $\B=\frac1q\Z$ with $q=10$.    
  The function is not extreme, but is affine imposing in $\Iedge$ with three
  slopes: a positive slope, a negative slope, and zero slope.  The intervals
  of positive slope are all connected in $\G$, as well as the intervals of
  negative slope.  But not all the intervals of zero slope are connected in
  $\G$.  The zero slope intervals on the left side of~$f$ and the zero slope
  intervals on the right side of~$f$ form two separate connected components of the
  graph.%
  \quad
  (b)  Its polyhedral complex $\Delta\P_q$.  
  The polyhedral complex is shaded bright green wherever $\Delta\pi(x,y) = 0$.  
  }%
  \label{figure:polyhedralComplex}%
  \label{figure:flatFunction}%
\end{figure}%

To make this graph finite, we will use the periodicity of the function~$\pi$
and of the complex~$\P_q$ modulo~$\Z$.  By $\P_{q,\edge}/\Z$ we denote the set of
equivalence classes 
\begin{displaymath}
  \EqClass{I} = \{\, \tau_s(I) = I + s \st s\in\Z\,\}
\end{displaymath}
of proper intervals $I
\in\P_{q,\edge}$ modulo translations by integers~$s\in\Z$.  We can identify an
equivalence class with its unique representative that is a subinterval of~$[0,1]$.
Note that if $\pi$ is affine imposing on an interval~$I\in\P_q$, then it is
affine imposing on all intervals $J \in \EqClass{I}$. 

\begin{definition}\label{def:graph}
  Let $\G = \G(\Iedge/\Z,\E)$ be the undirected graph with node set $\Iedge/\Z$
  and edge set $\E$ where $\{\EqClass{I}, \EqClass{J}\} \in \E$ if and only if
  $[I]\neq[J]$ and for some $1$-face $E \in \Delta\P_q$ and some face
  $F\in\Delta\P_q$ with $E\subseteq F$ and $\Delta\pi_F(x,y)=0$ for
  $(x,y)\in\relint(E)$ and some $v \in \tfrac{1}{q} \Z$ we have:
  \begin{itemize}
  \item (Case 1:) $I = p_1(E)$, $J = p_2(E)$, $\{v\} = p_3(E)$ (and thus $J = \rho_v(I)$
    and $I = \rho_v(J)$), or
  \item (Case 2:) $I = p_1(E)$, $\{v\} = p_2(E)$, $J = p_3(E)$ (and thus $J =
    \tau_v(I)$), or
  \item (Case 3:) $J = p_1(E)$, $\{v\} = p_2(E)$, $I = p_3(E)$ (and thus $I =
    \tau_v(J)$).
  \end{itemize}
\end{definition}

\begin{remark}
  This is well-defined because of the periodicity of~$\pi$
  modulo~$\Z$.  Indeed, let $\EqClass{I} = \EqClass{I'}$ and $\EqClass{J} =
  \EqClass{J'}$.  Let $I'=I+s$ and $J'=J+t$ with $s,t\in\Z$. 
  Suppose $I = p_1(E)$, $J = p_2(E)$, $\{v\} = p_3(E)$ for some $E,F\in\Delta\P_q$
  where $E$ is a $1$-face 
  with $E\subseteq F$ and $\Delta\pi_F=0$ on
  $\relint(E)$ and some $v\in\frac1q\Z$.
  Define $E' = E + (s,t)$ and $F' = F + (s,t)$.
  Then $E'$ is another $1$-face in $\Delta\P_q$, $F'$ is a face in
  $\Delta\P_q$ with $E' \subseteq
  F'$ and $\Delta\pi_{F'}=0$ on
  $\relint(E')$ and $I' = I + s = p_1(E')$, $J' = J+t = p_2(E')$, and $\{{v'}\} =
  p_3(E')$, where ${v'} = v + s + t \in\frac1q\Z$.
  
  On the other hand, suppose, without loss of generality, 
  $I = p_1(E)$, $\{v\} = p_2(E)$, $J = p_3(E)$ for some 
  $1$-face $E\in\Delta\P_q$ contained in a face $F\in\Delta\P_q$ 
  with $\Delta\pi_F=0$ on $\relint(E)$ and some $v\in\frac1q\Z$.
  Define $E' = E + (s, t - s)$ and $F' = F + (s, t-s)$.
  Then $E'$ is another $1$-face in $\Delta\P_q$, which is contained in the
  face $F'\in\Delta\P_q$, and 
  $\Delta\pi_{F'}=0$ on $\relint(E')$ with
  $I' = I + s = p_1(E')$, $\{{v'}\} = p_2(E')$, $J' = J + t = p_3(E')$, 
  where ${v'} = v - s + t$.
\end{remark}

\begin{lemma}
\label{lemma:pointLemma}
Let $\pi^1,\pi^2$ be valid functions such that $\pi = \frac12(\pi^1+\pi^2)$.
For $\theta = \pi, \pi^1$, or $\pi^2$, if $\theta$ is affine in $\intr(I)$,
and $\{\EqClass{I},\EqClass{J}\} \in \E$, then $\theta$ is affine in $\intr(J)$ as well with the
same slope.   
\end{lemma}
\begin{proof}
Since $\theta$ is affine in $\intr(I)$, let $c,b \in \R$ such that $\theta(x)
= cx + b$ for $x\in\intr(I)$. If $\{\EqClass{I},\EqClass{J}\} \in \E$, then
one of three cases could happen.

\noindent \textbf{Case 1.} 
$I = p_1(E)$, $J = \rho_a(I) = p_2(E)$, $\{a\} = p_3(E)$ for some $1$-face $E
\in\Delta\P_q$ contained in a face $F\in\Delta\P_q$ with $\Delta\pi_F=0$ on
$\relint(E)$ and some $a\in\frac1q\Z$. 

If $E = F$, then  for all $v\in \intr(J)$, $(a-v, v)\in\relint(E)=\relint(F)$ and thus
$0 = \Delta\pi_F(a-v, v) = \Delta\pi(a-v,v) = \pi(a - v) + \pi(v) - 
\pi(a) = 0$. Then, by Lemma~\ref{lem:tightness}, 
$\Delta\theta(a-v,v) = \theta(a - v) + \theta(v) - \theta(a) = 0$ for all $v\in \intr(J)$.  Since
$a - v = \rho_a(v) \in \intr(I)$ for all $v \in \intr(J)$, 
$\theta(a - v) = c(a - v) + b$ for
all $v \in \intr(J)$.  Then $\theta(v) = -\theta(a-v) + \theta(a) = c v - c a
-b + \theta(a)$ for $v\in
\intr(J)$, and therefore 
$\theta$ is affine in $\intr(J)$ with the same slope.

On the other hand, if $E \subsetneq F$, then $F$ is a $2$-face (triangle)
whose diagonal edge is~$E$. 
Let $\{a\} \subsetneq K = p_3(F)$. 
For all $v\in \intr(J)$, $(a-v,
v)\in F$ and thus by Lemma~\ref{lemma:tight-implies-tight},
\begin{equation}
  \label{eq:limit-xy-case-1}
  \lim_{\substack{(x,y) \to (a-v,v)\\ (x,y) \in
      \relint(F)}}\Delta\theta(x,y) 
  = \lim_{\substack{(x,y) \to (a-v,v)\\ (x,y) \in
      \relint(F)}} \Bigl(\theta(x) + \theta(y) - \theta(x+y)\Bigr) 
  = 0.
\end{equation}
Let $v\in\intr(J)$. Because $v \in \intr(J)$, the set 
$$ Z_v = \{\, z \st (z-v,v) \in \relint(F) \,\} 
= \intr(K) \cap (\intr(I) + v) $$
is a open subinterval of~$K$ with $a$ on its boundary. 
Thus we can specialize the limit~\eqref{eq:limit-xy-case-1} to the points
$(x,y)=(z-v,v)$ with $z\in Z_v$, 
\begin{equation}
  \label{eq:limit-z-case-1}
  \lim_{\substack{z \to a\\ z \in Z_v}} 
    \Bigl(\theta(z - v) + \theta(v) - \theta(z)\Bigr) 
  = 0.
\end{equation}
But since $a - v \in\intr(I)$, $\theta$ is continuous at $a - v$, and so 
$$\lim_{\substack{z \to a\\ z \in Z_v}} \theta(z - v) = \theta(a-v).$$
Since also $\theta(v)$ does not depend on~$z$, we find that 
the limit of~$\theta(z)$ exists and equals
\begin{equation}
  \label{eq:limit-case-1}
  \lim_{\substack{z \to a\\ z \in Z_v}} \theta(z) = \theta(a-v) + \theta(v).
\end{equation}
We claim that the limit~\eqref{eq:limit-case-1} is independent of the choice
of~$v \in\intr(J)$.  Indeed, let $v' \in\intr(J)$, then 
$Z_v$, $Z_{v'}$, and $ Z_v \cap Z_{v'} $ are open (nonempty) subintervals of~$K$
with $a$ on their boundaries, and thus
$$ \lim_{\substack{z \to a\\ z \in Z_v}} \theta(z)
= \lim_{\substack{z \to a\\ z \in Z_v \cap Z_{v'}}} \theta(z)
= \lim_{\substack{z \to a\\ z \in Z_{v'}}} \theta(z). $$
Let $L$ denote this limit. Then \eqref{eq:limit-case-1} implies $\theta(v) = -\theta(a-v) + L = c v - c a
- b + L$ for $v \in \intr(J)$, and therefore $\theta$ is affine in $\intr(J)$
with the same slope.\smallskip

\noindent \textbf{Case 2.} 
$I = p_1(E)$, $\{a\} = p_2(E)$, $J = \tau_a(I) = p_3(E)$ for some $1$-face $E
\in\Delta\P_q$ contained in a face $F\in\Delta\P_q$ with $\Delta\pi_F=0$ on
$\relint(E)$ and some $a\in\frac1q\Z$.  For brevity, we only prove the case
where $E = F$. 
Then 
$\pi(w -  a) + \pi(a) = \pi(w)$ 
and thus, by Lemma~\ref{lem:tightness}, 
$\theta(w -  a) + \theta(a) = \theta(w)$ for all $w \in \intr(J)$. Again,
since $w - a = \tau_a^{-1}(w) \in \intr(I)$ for all $w \in \intr(J)$, we have $\theta(w - a) = c(w - a)
+ b$ for all $w \in \intr(J)$. Then we obtain $\theta(w) = c w - c a + b + 
\theta(a)$ for $w\in\intr(J)$, and therefore $\theta$ is affine in $\intr(J)$
with the same slope.  

\noindent \textbf{Case 3.} 
$J = p_1(E)$, $\{a\} = p_2(E)$, $I = \tau_a(J) = p_3(E)$ for some $1$-face $E
\in\Delta\P_q$ contained in a face $F\in\Delta\P_q$ with $\Delta\pi_F=0$ on
$\relint(E)$ and some $a\in\frac1q\Z$.  For brevity, we only prove the case
where $E = F$. 
Then $\pi(u) + \pi(a) = \pi(u+a)$ and thus, by Lemma~\ref{lem:tightness}, 
$\theta(u) + \theta(a) = \theta(u+a)$ for all $u \in\intr(J)$.  Since $u + a =
\tau_a(u) \in \intr(I)$ for all $u \in \intr(J)$, we have $\theta(u + a) = c(u
+ a) + b$ for all $u\in \intr(J)$.  Then we obtain $\theta(u) = \theta(u + a)
- \theta(a) = c u + c a + b - \theta(a)$ for $u\in\intr(J)$, and therefore $\theta$ is affine in
$\intr(J)$ with the same slope.
\end{proof}

\paragraph{Intervals connected to covered intervals.}
For each $I \in \Iedge$, let $\G_I$ be the connected component of $\G$
containing $\EqClass{I}$.   
We now define the set of all intervals connected to covered intervals in
the graph~$\G$, 
$$
\S^2_{q,\edge} =\bigl\{\,J \in \Iedge \bigst \EqClass{J} \in \G_I \text{ for some } I \in \I^2_{q,\edge}\,\bigr\}.
$$
This is a set of intervals on which $\pi$ is affine imposing.
(Actually, we will see below that it is exactly
the set of intervals on which $\pi$ is affine imposing; this will follow from
Lemma~\ref{lemma:not-extreme}.) 
\begin{theorem}
  \label{theorem:AI}
  \begin{enumerate}[\rm(i)]
  \item $\pi$ is affine imposing on $\S^2_{q,\edge}$.
  \item Moreover, let
    $\pi^1,\pi^2$ be valid functions such that $\pi =
    \frac12(\pi^1+\pi^2)$. Then, for $\theta = \pi, \pi^1, \pi^2$ and for each
    interval $I \in \S^2_{q,\edge}$, the function $\theta$ is affine in $\intr(J)$ with the
    same slope for every interval $J \in \G_I$.
  \end{enumerate}
\end{theorem}
\begin{proof}
From Lemma \ref{lemma:pointLemma}, it follows that if $\pi$ is affine imposing in $I$ and $\{\EqClass{I},\EqClass{J}\} \in \E$, then $\pi$ is affine imposing in $J$.
Let $J \in \S^2_{q,\edge}$.   Then $J$ must be in a connected component containing an interval $I \in \I^{2}_{q,\edge}$.  
By induction on the length of a path between $\EqClass{I}$ and $\EqClass{J}$
in $\G$,
$\pi$ is affine imposing in each interval whose equivalence 
class is a node of the connected component~$\G_I$, and therefore is affine
imposing in $J$.   
We conclude that $\pi$ is affine imposing in $\S^2_{q,\edge}$. 
It follows directly from Lemma \ref{lemma:pointLemma} that for each interval $I \in \S^2_{q,\edge}$, $\theta$ has the same slope in $J$ for every interval $J \in \G_I$ and $\theta = \pi, \pi^1, \pi^2$.
\end{proof}

\begin{corollary}
  \label{cor:AI-everywhere}
  Suppose $\S^2_{q,\edge} = \Iedge$.  Then $\pi$ is affine imposing in $\Iedge$.
\end{corollary}

The function in Figure \ref{figure:flatFunction} illustrates how intervals can
be connected in~$\G$ and that intervals on which $\pi$ has the same slope are
not necessarily connected. 

\subsection{Non-extremality by equivariant perturbation}
\label{sec:non-extreme-by-perturbation}

In this subsection, we will prove the following result.

\begin{lemma}
  \label{lemma:not-extreme}
  If $\S^2_{q,\edge} \neq \Iedge$, then $\pi$ is not extreme.
\end{lemma}


\begin{proof}
Let $I \in \Iedge\setminus \S^2_{q,\edge}$.  
This is an interval for which $\pi$ is not known to be affine imposing.
Consider the set 
\begin{displaymath}
  R = \bigcup_{\substack{J\in \P_{q,\edge}\\ [J] \in \G_I}} \intr(J),
\end{displaymath}
a union over all
intervals whose equivalence classes are connected to~$\EqClass{I}$ in the
graph~$\G$.   
Note that $\pi$ is not known to be affine imposing over any of these
intervals~$J$.  Indeed we will show that we can perturb the function~$\pi$
simultaneously over all these intervals using a piecewise linear function with
breakpoints in~$\frac1{4q}\Z$.

To this end, let $\psi$ be the $\Gamma$-equivariant function of
Lemma~\ref{lemma:our-equivariant-psi}.  This is a continuous piecewise linear
function with 
breakpoints in~$\frac1{4q}\Z$, which is zero on all of the breakpoints in
$\frac1{2q}\Z$. 
Let $\bar \pi = \delta_R  \cdot\psi$ where $\delta_R$ is the indicator
function for the set $R$. 
By multiplying with the indicator function, we restrict the
$\Gamma$-equivariant function to the set~$R$; outside of~$R$, the resulting
function~$\bar\pi$ is zero.  Because the boundary of~$R$ consists of
breakpoints in $\frac1q\Z$, where $\psi$ is zero, the function~$\bar\pi$ is
continuous. 
Let 
$$
\epsilon = \min \Bigl\{\,\Delta\pi_{\hat F}(x,y)\neq 0 \mathrel{\Big|} {\hat
  F} \in \Delta\P_{\frac1{4q}\Z}, \ \hat F\subseteq[0,1]^2, \ (x,y) \in \verts(\hat F)\,\Bigr\}.
$$
Note that $\epsilon$ is well-defined and positive because it is a minimum over
a finite number of positive numbers. Here we consider~$\pi$ as a piecewise
linear function with breakpoints in~$\frac1{4q}\Z$ and use the fact that $\pi$
is subadditive. 
We will show that for 
$$\pi^1 = \pi + \tfrac{\epsilon}{3}\bar \pi, \  \ \pi^2 = \pi - \tfrac{\epsilon}{3}\bar \pi,$$
that $\pi^1, \pi^2$ are minimal valid functions, and hence $\pi$ is not extreme.  We will show this just for $\pi^1$ as the proof for $\pi^2$ is the same.

First, observe that $\pi^1$ is periodic modulo~$\Z$. 

Also, since $\psi(0) =  0$ and $\psi(f) = 0$, we see that $\pi^1(0) =  0$ and $\pi^1(f)= 1$.   

We want to show that $\pi^1$ is symmetric and subadditive. We will do this by
analyzing the function $\Delta\pi^1(x,y) = \pi^1(x) +  \pi^1(y) -
\pi^1(x+ y)$ and showing that $\Delta\pi^1(x,y)\geq0$ for all $x,y\in\R$.  Since $\psi$ is piecewise linear over $\frac1{4q}\Z$,
$\pi^1$ is also piecewise linear with breakpoints in $\frac1{4q}\Z$.  By
Theorem~\ref{minimality-check},
we only need to check $\Delta\pi^1(x,y)\geq0$ for all vertices $(x,y)$ of the
complex $\Delta\P_{4q}$. 

Let $\hat{I}, \hat{J}, \hat{K} \in \Ipoint[4q] \cup \Iedge[4q]$, such that ${\hat F} = \{\,(x,y)
\st x \in \hat{I}, y \in \hat{J}, x+ y \in \hat{K}\,\} \in \Delta\P_{4q}$ is non-empty. Let
$\Delta\pi^1_{\hat F}(u,v) = \pi^1_{\hat{I}}(u) +\pi^1_{\hat{J}}(v) - \pi^1_{\hat{K}}(u+ v)$. Let
$(u,v) \in \verts(\hat F)$.  

In the following, we consider two cases: If $\Delta\pi_{\hat F}(u,v) > 0$ (strict subadditivity), we make use of our
choice of $\epsilon$ to show that $\Delta\pi^1_{\hat F}(u,v) \geq 0$.  On the
other hand, if $\Delta\pi_{\hat F}(u,v)= 0$ (additivity), we show that
$\Delta\pi^1_{\hat F}(u,v) = 0$.  This will prove two things.  First,
$\Delta\pi^1(x,y) \geq 0$ for all $x,y \in \R$, and therefore $\pi^1$ is
subadditive.  Second, since $\pi$ is symmetric,  if ${\hat F} \subset
\{\,(x,y)\st x+ y \equiv f \pmod1\,\}$, then $\Delta\pi_{\hat F}(x,y) = 0$ for all
$(x,y)\in \verts(\hat F)$,  which would imply that $\Delta\pi_{\hat F}^1(x,y)
= 0$ for all $(x,y) \in \verts(\hat F)$, proving that $\pi^1$ is
symmetric.\smallbreak

\noindent \textbf{Strictly subadditive case.} First, if $\Delta\pi_{\hat
  F}(u,v) > 0$, then $\Delta\pi_{\hat F}(u,v) \geq \epsilon$.  
Therefore, since $|\bar\pi|\leq 1$,
$$\Delta\pi_{\hat F}^1(u,v) \geq \pi_{\hat{I}}(u) - \epsilon/3 + \pi_{\hat{J}}(v) -
\epsilon/3 - \pi_{\hat{K}}(u+ v) - \epsilon/3 = \Delta\pi_{\hat F}(u,v) -
\epsilon \geq 0.$$ %

\begin{figure}[t]
\centering
(a)\quad\vcenteredhbox{\includegraphics[scale = 0.8]{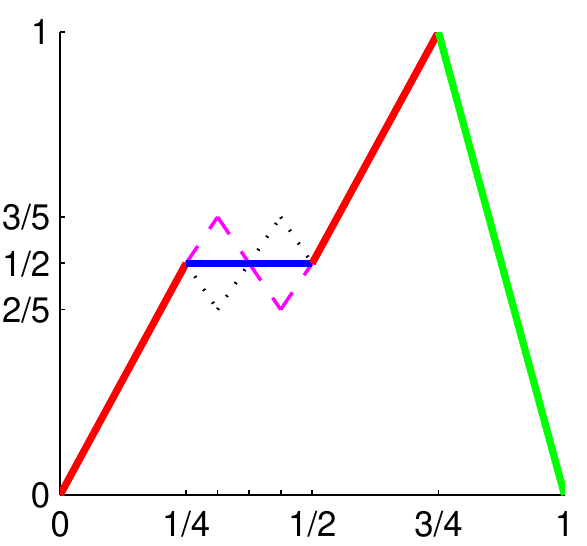}}
(b)\quad\vcenteredhbox{\scalebox{0.65}{\ifpdf
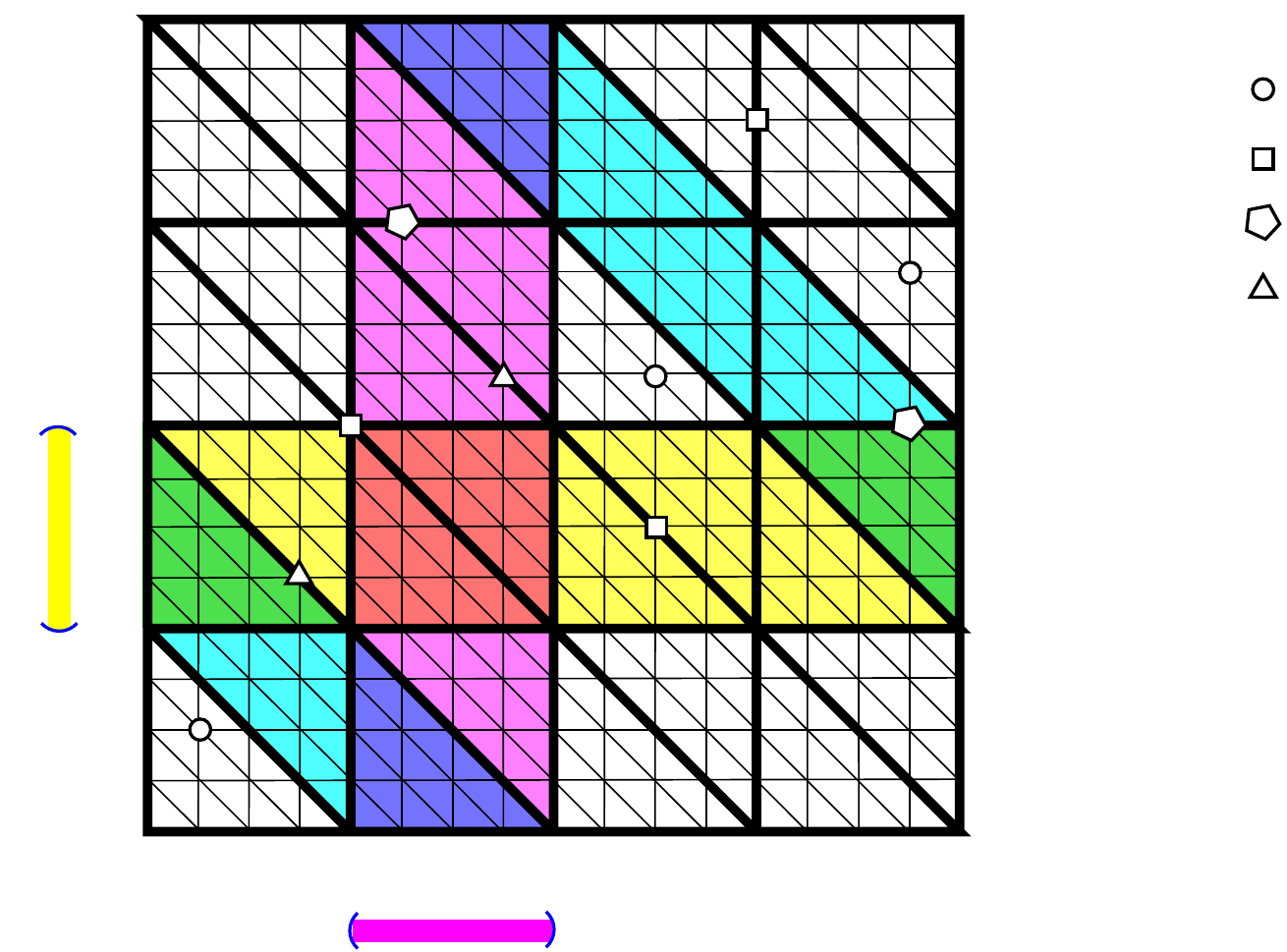
\else
\input{figures/notAInotExtreme4new.pstex_t}
\fi}}
\caption{Cases within the proof of Lemma \ref{lemma:not-extreme} on a sample
  non-extreme function.  
  (a)~The function~$\pi$ is piecewise linear on~$\P_q = \P_4$
  with three slopes (\emph{red, blue, green solid lines}).  We construct the perturbed functions~$\pi^1$ (\emph{dashed lines})
  and $\pi^2$ (\emph{dotted lines}). 
(b)~The
  two-dimensional complex $\Delta \P_q=\Delta\P_4$ (\emph{thick black lines}) and its refinement
  $\Delta\P_{4q}=\Delta\P_{16}$ (\emph{thin black lines}).  The
  union $R$ of open sets is the open interval $(\tfrac14, \tfrac12)$ in this
  example.  The regions of points $(u,v)$ with $u \in R$ (\emph{magenta}), $v \in R$ (\emph{yellow}), or $u + v \in
  R$ (\emph{cyan}) are \emph{shaded} using the subtractive color model; thus we see the points 
  with $u,v\in R$ (\emph{red}), $u, u + v\in R$ (\emph{blue}), $v, u + v\in R$
  (\emph{green}). 
  In Subcase 1 (\emph{circles}), $u,v,u + v \notin
  R$ (\emph{unshaded white regions}).  In Subcase 2 (\emph{squares}), $u,v \in \tfrac{1}{2q} \Z$.  In Subcase 3, we are
  not in Subcases 1 or~2;  we show that $(u,v)$ lies in the relative interior
  of a $1$-dimensional face~$F$. In Subcase 3a (\emph{pentagons}), $F \subset \{\,(x,y) \st y =
  v\,\}$ and $v \in \tfrac{1}{q} \Z$.  
  In Subcase 3b (\emph{triangles}), $F \subset \{\,(x,y) \st x + y = u + v\,\}$ and 
  $u + v \in \tfrac{1}{q} \Z$.
}\label{figure:notAInotExtreme}
\end{figure}

\noindent \textbf{Additive case.} Next, we will show that if $\Delta\pi_{\hat
  F}(u,v) = 0$, then $\Delta\pi_{\hat F}^1(u,v) = 0$.    
Suppose that $\Delta\pi_{\hat F}(u,v) = 0$.  We will proceed by cases.  See
Figure \ref{figure:notAInotExtreme} for an illustration of these
cases.\smallskip

\noindent\textbf{Subcase 1.}  Suppose $u,v, u + v \notin R$.  Then
$\delta_R(u) = \delta_R(v) = \delta_R(u + v) = 0$, and $\Delta\pi^1_{\hat
  F}(u,v) = \Delta\pi_{\hat F}(u,v) \geq 0$.\smallbreak

\noindent\textbf{Subcase 2.} Suppose $u,v \in \tfrac{1}{2q} \Z$.  Then $u + v \in
\tfrac{1}{2q}\Z$ and, by Lemma~\ref{lemma:our-equivariant-psi}~(i),
$\psi(u) = \psi(v) = \psi(u + v) = 0$.  Thus $\Delta\pi_{\hat F}^1(u,v) =
\Delta\pi_{\hat F}(u,v) \geq 0$.\smallbreak

\noindent \textbf{Subcase 3.} Suppose we are not in subcases 1 or 2.  That is, suppose $\Delta\pi_{\hat F}(u,v) = 0$, not both $u,v$ are in $\smash[b]{\tfrac{1}{2q} \Z}$, and at least one of $u,v,u + v$ is in $R$.
Since $\Delta\pi^1(x,y)$ is symmetric in $x$ and~$y$, without loss of generality, since not both $u,v$ are in $\tfrac{1}{2q} \Z$, we will assume that $u \notin \tfrac{1}{2q} \Z$. 

There exists a unique face $F \in
\Delta\P_q$ with $\relint(\hat F) \subseteq \relint(F)$. Then
$\Delta\pi_F(x,y) = \Delta\pi(x,y) = 
\Delta\pi_{\hat{F}}(x,y)$ for all $(x,y)\in\relint(\hat{F})$.  
This implies 
$\Delta\pi_F(x,y) = \Delta\pi_{\hat{F}}(x,y)$ for all $(x,y)\in\hat{F}$.
There is a unique face $E \in \Delta\P_q$ with $E\subseteq F$ and
$(u,v)\in\relint(E)$. 
Since $u \notin \tfrac{1}{2q} \Z$, $(u,v) \notin \verts(\Delta\P_q)$, and thus
$E$ is $1$-dimensional or $2$-dimensional.
Since $\pi$ is subadditive,
$\Delta\pi_{F}(x,y) = \Delta\pi(x,y) \geq 0$ for all $(x,y) \in \relint(F)$,
and thus $\Delta\pi_{F}(x,y) \geq 0$ for all $(x,y) \in F$. 
Now since the affine function $\Delta\pi_{F}$ equals $0$ at the point
$(u,v)\in \relint(E) \subseteq F$,  
we have 
\begin{displaymath}
  \Delta\pi_{F}(x,y) = 0 \quad\text{for all}\quad (x,y)\in\relint(E). 
\end{displaymath}
If $E$ were a $2$-dimensional face, then
$u\in \intr(I)$, $v \in \intr(J)$, and $u + v\in \intr(K)$ for some $I,J,K \in \I^2_{q,\edge}$,
thus $u,v,u + v\notin R$, which is subcase~1. Therefore, we can assume
that $E \in \Delta\P_q$ is a $1$-dimensional face and hence a subset of one of the three following hyperplanes: $x = u$, $y = v$, or $x +y = u  + v$.

Since $u \notin  \tfrac{1}{2q}\Z \supset \tfrac{1}{q} \Z$, $E$ cannot be a
subset of $x = u$ because it is not a defining hyperplane of the polyhedral
complex $\Delta\P_q$. Observe that for any $x \in \R$ with $x \notin
\tfrac{1}{q} \Z$, there is a unique $I_x \in \Iedge$ such that $x \in \intr(I_x)$.
Since $u \notin \tfrac{1}{q} \Z$, there exists a unique interval $I_u \in
\Iedge$ such that $u \in \intr(I_u)$.  There are two possible subcases.   \smallbreak

 \indent \textbf{Subcase 3a.}  $E \subset \{\,(x,y) \st y = v\,\}$ and $v \in \tfrac{1}{q} \Z$.\\
 Since $v \in \tfrac{1}{q} \Z, u \notin  \tfrac{1}{2q}\Z$, it follows that $u
 + v \notin \tfrac{1}{q} \Z$, and there is a unique interval $I_{u + v} \in
 \Iedge$ containing $u + v$ in its interior.   
Then $p_1(E) = I_u$, $p_2(E) = \{v\}$, and $p_3(E) = {\tau}_v(I_u) = I_{u+v}$, and 
$\Delta\pi_{F}(x,y) = 0$ for $(x,y)\in \relint(E)$.
Therefore, by Definition~\ref{def:graph}, $\{\EqClass{I_u}, \EqClass{I_{u + v}}\} \in
\E$ and $\delta_R(u) = \delta_R(u + v)$.  Since $v \in \tfrac{1}{q} \Z$, we have
$\psi(v) = 0$ and $\psi(u) = \psi({\tau}_v(u)) = \psi(u + v)$ by
Lemma~\ref{lemma:our-equivariant-psi}~(iii).  It follows that
$\Delta\bar\pi(u,v) = \bar \pi(u) +
\bar\pi(v) - \bar\pi(u + v) = 0$, and therefore, since $\bar\pi$ is
continuous,
$\Delta\pi_{\hat F}^1(u,v)
= \Delta\pi_{\hat F}(u,v) + \Delta\bar\pi(u,v)
= \Delta\pi_{\hat F}(u,v) = 0$.\smallbreak

\indent \textbf{Subcase 3b.}  $E \subset \{\,(x,y) \st x + y = u + v\,\}$ and $u + v \in \tfrac{1}{q} \Z$.\\
 Since $u + v \in \tfrac{1}{q} \Z, u \notin \tfrac{1}{2q} \Z$, it follows that
 $v \notin \tfrac{1}{q} \Z$, and there is a unique interval $I_{v}\in \Iedge$
 containing $v$ in its interior.   
 Then $p_1(E) = I_u$, $p_2(E) = \rho_{u + v}(I_u) = I_v$, $p_3(E) = \{u+v\}$, 
 and $\Delta\pi_{F}(x,y) = 0$ for $(x,y)\in \relint(E)$.
 Therefore, by Definition~\ref{def:graph}, $\{\EqClass{I_u}, \EqClass{I_{v}}\}
 \in \E$ and $\delta_R(u) = \delta_R(v)$.   
 Since $u + v \in \tfrac{1}{q} \Z$, we have $\psi(u+v) = 0$ and $\psi(u) = -
 \psi({\rho}_{u + v}(u)) = - \psi(v)$ by
 Lemma~\ref{lemma:our-equivariant-psi}~(ii).  It follows that
 $\Delta\bar\pi(u,v) = \bar \pi(u) +
 \bar\pi(v) - \bar\pi(u + v) = 0$, and therefore, since $\bar\pi$ is
 continuous, $\Delta\pi_{\hat F}^1(u,v) = \Delta\pi_{\hat F}(u,v) + \Delta\bar\pi(u,v)
 = \Delta\pi_{\hat F}(u,v) = 0$.\medskip

 We conclude that $\pi^1$ (and similarly $\pi^2$) is subadditive and
 symmetric, and therefore by Theorem~\ref{minimality-check} minimal and hence
 valid.  Therefore $\pi$ is not extreme. 
\end{proof}

\begin{remark}
To show that $\pi$ is not extreme in the above lemma, the perturbation
function $\psi$ need not be piecewise linear.  This choice was  made to
simplify the proof.  In fact, any $\Gamma$-equivariant function
$\tilde\psi\neq 0$ constructed with Lemma \ref{lemma:phi} with $|\tilde\psi| <
|\psi|$ suffices. 
Compare Figures \ref{fig:equivariant-general} and \ref{figure:phi}.
\end{remark}

However, the specific form of our function~$\psi$ as a piecewise linear
function with breakpoints in $\frac1{4q}\Z$ (Lemma~\ref{lemma:our-equivariant-psi}~(iv)) implies the following corollary.

\begin{corollary}
\label{corollary:AIG4}
If $\pi$ is not affine imposing over $\Iedge$, then there exist distinct minimal
$\pi^1, \pi^2$ that are piecewise linear with breakpoints in $\frac1{4q}\Z$ such that $\pi = \tfrac{1}{2}(\pi^1+ \pi^2)$.  
\end{corollary}

\subsection{Extremality and non-extremality by a system of linear equations}
\label{section:system-use}

Now we are able to prove that the finite system of linear equations,
introduced in section \ref{section:system}, can decide extremality of~$\pi$ if
the set $B$ of breakpoints is chosen appropriately.

\begin{theorem}\label{thm:1/4q test}
Let $\pi$ be a piecewise linear minimal valid function (possibly
discontinuous) whose  
breakpoints are rational with least common denominator~$q$.  
Then $\pi$ is extreme if and only if the system of
equations~\eqref{equation:system} with $\B = \tfrac{1}{4q}\Z$
has a unique solution.
\end{theorem}
\begin{proof}
  The forward direction is the contrapositive of
  Theorem~\ref{theorem:systemNotUnique}, applied to the set of breakpoints $\B
  = \frac1{4q}\Z$. 
  For the reverse direction,   
  we assume that the the system of
  equations~\eqref{equation:system} with $\B = \tfrac{1}{4q}\Z$
  has a unique solution.
  Suppose that there exist distinct valid functions
  $\pi^1, \pi^2$ that are piecewise linear with breakpoints in $\frac1{4q}\Z$
  such that $\pi = \tfrac{1}{2}(\pi^1+ \pi^2)$.   
  By Lemma~\ref{lem:minimality-of-pi1-pi2}, $\pi^1$ and $\pi^2$ are
  minimal. 
  Then $\pi^1$ and $\pi^2$ are solutions to~\eqref{equation:system} with $\B =
  \tfrac{1}{4q}\Z$, a contradiction.  
  Thus, by the contrapositive of Corollary~\ref{corollary:AIG4}, $\pi$ is
  affine imposing in $\Iedge$.  
  Then $\pi$ is also affine imposing on $\Iedge[4q]$ since it is a finer interval set.  
  By Corollary~\ref{corollary:iff}, since $\pi$ is affine imposing in
  $\Iedge[4q]$ and the system of 
equations~\eqref{equation:system} with $\smash[t]{\B = \frac1{4q}\Z} $ has a
unique solution, $\pi$ is extreme.   
\end{proof}

\subsection{Proofs of Theorems~\ref{thm:main} and~\ref{thm:finite_group}}
\label{sec:proof-intro-thm}

We are now ready to present the proofs of Theorems~\ref{thm:main} and~\ref{thm:finite_group}.

\begin{proof}[Proof of Theorem~\ref{thm:main}]
  Given the piecewise linear function $\pi$, the algorithm performs the
  following test.  Test if the system~\eqref{equation:system} with $\B = \frac1{4q}\Z$ has a unique solution, where $q$ is the least common
  denominator of the breakpoints of $\pi$. If yes, then report that $\pi$ is
  extreme; else, report that $\pi$ is not extreme.
Theorem~\ref{thm:1/4q test} guarantees the correctness of this algorithm.

Finally, observe that the number of variables and constraints of
system~\eqref{equation:system} is bounded polynomially in~$q$, and it
can be written down and solved in a number of elementary operations over the reals that is
bounded by a polynomial in $q$. 
\end{proof}

\begin{proof}[Proof of Theorem~\ref{thm:finite_group}]
We first show that if $\pi$ is extreme, then $\pi|_{\frac1{4q}\Z}$ is extreme for $R_f(\frac1{4q}\Z, \Z)$. If not, then there exist two distinct minimal functions $\bar \pi^1, \bar \pi^2$, both functions from $\frac1{4q}\Z$ to $\R_+$, such that $\pi|_{\frac1{4q}\Z} = \frac12 (\bar \pi^1 + \bar \pi^2)$. Let $\pi^i$ be the linear interpolation of $\bar\pi^i$, $i = 1,2$. It can be verified that $\pi^i$ is minimal because $\bar \pi^i$ is minimal. This contradicts the extremality of $\pi$.

Now we show that if $\pi|_{\frac1{4q}\Z}$ is extreme for $R_f(\frac1{4q}\Z, \Z)$, then
$\pi$ is extreme. If $\pi$ is not extreme then by Theorem~\ref{thm:1/4q test}
the system of equations~\eqref{equation:system} does not have a unique
solution. Then we can construct $\pi^1$ and $\pi^2$ as in the proof of
Theorem~\ref{theorem:systemNotUnique} using the non-trivial element in the
kernel of~\eqref{equation:system}, such that $\pi = \frac12 (\pi^1 + 
\pi^2)$. Since $\pi$ is a continuous piecewise linear function, $\smash{\limsup_{h\to
  0}\lvert \frac{\pi(h)}{h} \rvert < \infty}$.
Theorem~\ref{Theorem:functionContinuous} then tells us that $\pi^1$ and $\pi^2$ both have to be continuous, and by construction have breakpoints in $\frac1{4q}\Z$. Thus, since $\pi^1 \neq \pi^2$, there exists a breakpoint $v \in \frac1{4q}\Z$ such that $\pi^1(v) \neq \pi^2(v)$. Thus, $\pi^1|_{\frac1{4q}\Z}$ and $\pi^2|_{\frac1{4q}\Z}$ are two distinct minimal valid functions for $R_f(\frac1{4q}\Z, \Z)$. Moreover, since $\pi = \frac12 (\pi^1 + \pi^2)$, we have that $\pi|_{\frac1{4q}\Z} = \frac12 \bigl(\pi^1|_{\frac1{4q}\Z} + \pi^2|_{\frac1{4q}\Z}\bigr)$. This contradicts the extremality of $\pi|_{\frac1{4q}\Z}$.
\end{proof}

\section{The irrational case: A new principle for proving extremality}
\label{sec:irrational}

In this section we investigate the properties of a family of piecewise linear
continuous functions, periodic modulo~$\Z$, which we illustrate in Figure
\ref{figure:irrationalFunction}.  Each of these functions has three slopes, one of which
is zero, and has translation points $a_0, a_1, a_2$ such that $\pi(a_i) +
\pi(x) = \pi(a_i + x)$ for $i=0,1,2$ and $x$ in a certain interval $[A,A_i]$.
When certain parameters are chosen appropriately, we will show that this
function is extreme.  In doing so, we showcase a new type of proof for a
function to be extreme.  

\subsection{Function requirements}
\label{subsection:requirements}
Here we explain restrictions that we require of some of the breakpoints of our
function; see also Figure~\ref{figure:Zoom}. 
\begin{assumption}\label{assum:requirements}
   Let $a_0, a_1, a_2, t_1, t_2, f, A, A_0, A_1, A_2 \in(0,1)$ such that the
   following hold: 
   \begin{enumerate}[\rm(i)]
   \item The numbers $t_1, t_2$ are linearly independent over $\Q$.
   \item We have $a_1 = a_0 + t_1, a_2 = a_0 + t_2$, and $0 < a_0 < a_1 < a_2
     < A < f/2$,  
   \item We have $a_i + A = f-A_i$ and $A_0 > A_1 > A_2 \geq \tfrac{A_0 + A}{2} > A \geq 0$.
  \end{enumerate}
\end{assumption}
Let $x_0$ be the midpoint of $[A,A_0]$, that is, $x_0 = (A + A_0)/2$.
Observe that, for $i=1,2$, 
$$
t_i =a_i - a_0 =  A_0 - A_i < A_0 - \frac{A_0 + A}{2} = \frac{A_0 - A}{2}.
$$
We then have $x_0 \pm t_i \in [A,A_0]$ for $i=1,2$.
\begin{figure}[t]
\centering
\ifpdf
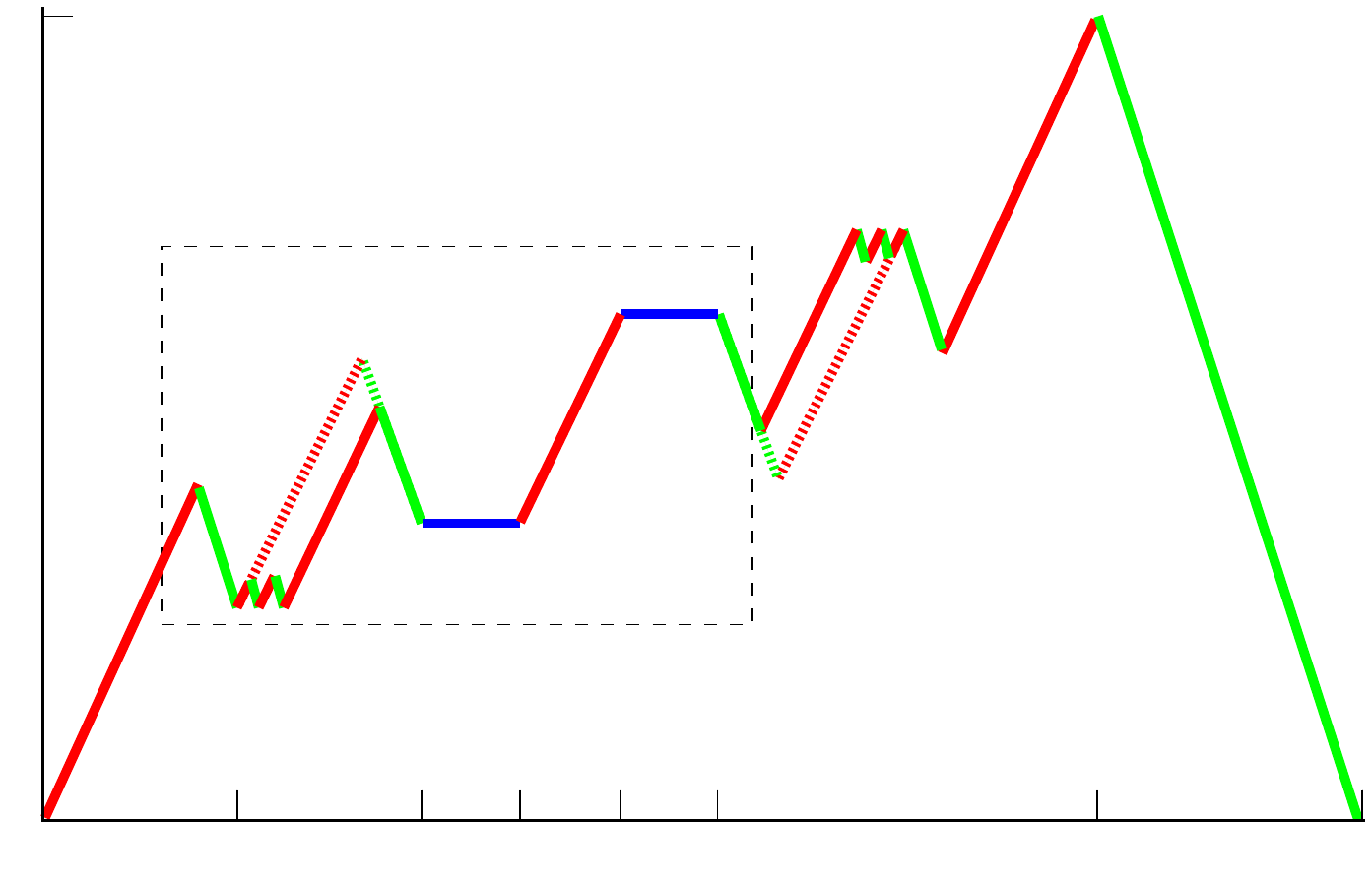
\else
\input{figures/irrationalFigureMediumFixed.pstex_t}
\fi
\caption{This function $\pi$ is extreme.  To show this, we use three
  translations that combine to create a dense set of relations that cause
  $\pi$ to be affine imposing where $\pi$ has slope $0$.  This argument
  crucially relies on the fact that one of the translations is irrational.  Using standard techniques, we show $\pi$ is affine imposing on all other intervals and then show that $\pi$ is the unique solutions to a system of equations.  Surprisingly, if all the points are rational, then this function is not extreme. 
We construct this function in two steps.  In Step 1, we describe the function
above with dotted lines.  In Step 2, we add on the two extra zig-zags.  See
Figure~\ref{figure:Zoom} for more details (dashed box) of this function.}
\label{newExtremeFunction}
\label{figure:irrationalFunction}
\end{figure}

\begin{figure}[t]
\centering
\ifpdf
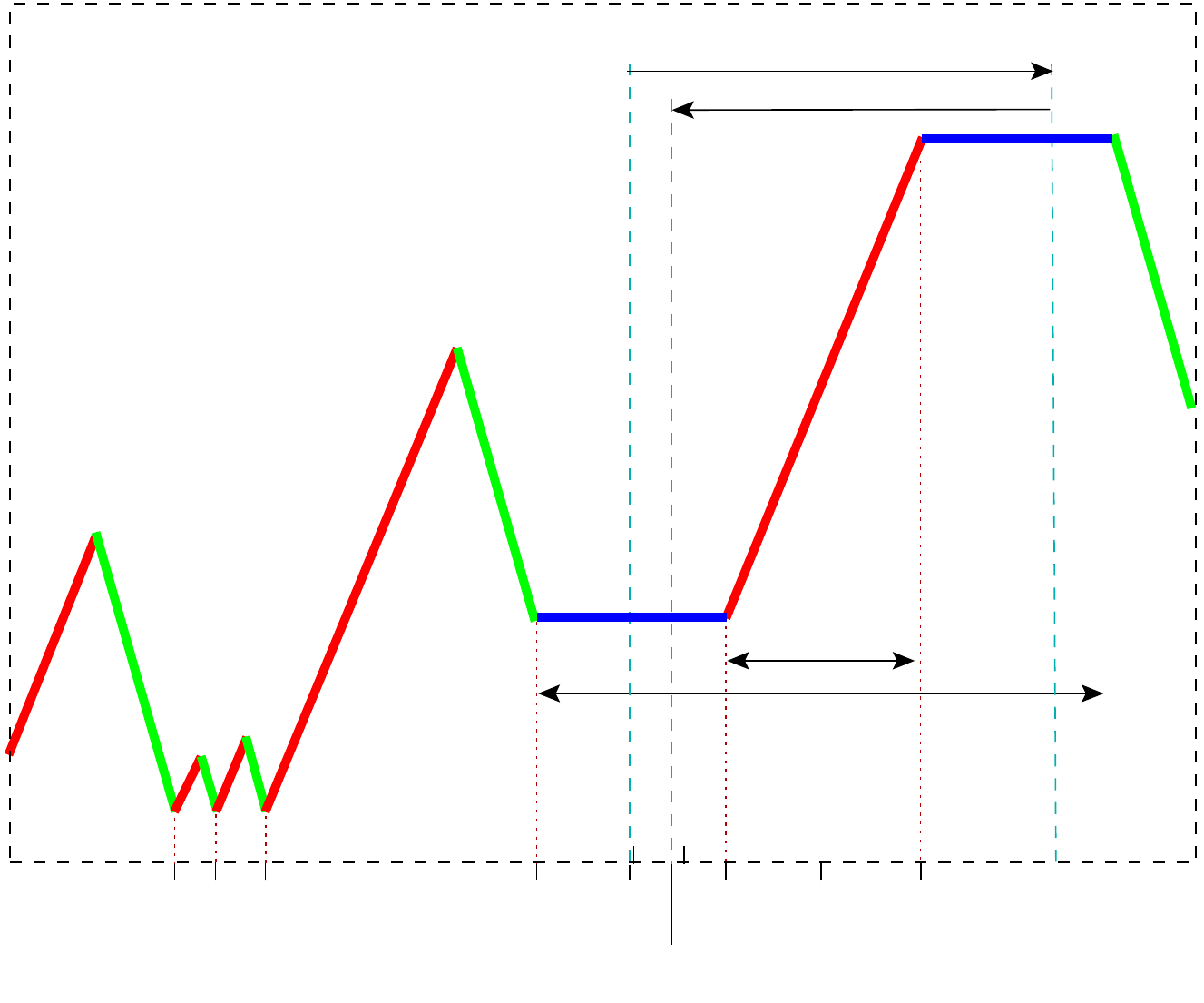
\else
\input{figures/irrationalFigureZoomArrows.pstex_t}
\fi
\caption{A zoomed in picture of Figure \ref{newExtremeFunction} to show in
  further detail the variables discussed in this section.  For our
  construction to work, we choose $a_i$ and $f - a_i$, for either $i=1$ or
  $i=2$, to be irrational numbers, and all other breakpoints to be rational.
  It follows that $t_1 = a_1 - a_0$ and $t_2 = a_2-a_0$ are linearly
  independent over $\Q$.
  The arrows indicate the action of the translations~$\tau_{a_1}^{}$,
  $\tau_{a_0}^{-1}$ and the reflection~$\rho_f$.
}
 \label{figure:Zoom}
\end{figure}

\smallbreak

Let $\Gamma = \langle \rho_f,\tau_{ a_0}, \tau_{a_1}, \tau_{a_2}\rangle$, as
defined in section \ref{section:Reflection-group-lemma},  and consider the
orbit $$X = \Gamma(x_0) = \{\,\gamma(x_0 )\st \gamma \in \Gamma\,\}.$$  
 From Lemma \ref{lemma:semidirect}, $\Gamma = \Tau \rtimes \langle \rho_f
 \rangle$ where $\Tau = \langle \tau_{a_0}, \tau_{a_1}, \tau_{a_2}\rangle
 = \langle  \tau_{a_0}^{},  \tau_{a_0}^{-1} \tau_{a_1}^{}, \tau_{a_0}^{-1}
 \tau_{a_2}^{}\rangle$.   
 Let $\tau_a \in \Tau$ be a translation and observe that
   $$
   \tau_{a_0}(\tau_{a}(x_0)) = a_0 +  (x_0 +  a) = (a_0 +  x_0) +  a = (f-x_0) +  a = f - (x_0 - a) = \rho_f^{}(\tau_{a}^{-1}(x_0)).
   $$
   Therefore, the translation $\tau_{a_0}$ is redundant in this orbit and we
   can rewrite $X$ with one fewer translation as $X = \{\, \gamma(x_0) \st
   \gamma \in \langle \rho_f, \tau_{a_0}^{-1} \tau_{a_1}^{}, \tau_{a_0}^{-1}
   \tau_{a_2}^{}\rangle\,\}$, or more simply as 
$$     
X = (x_0 + \Lambda) \cup (\rho_f(x_0) +\Lambda)
$$
where $\Lambda = {\langle t_1, t_2 \rangle}_\Z$.  The key element here is that $\Lambda$ is dense in $\R$ because $t_1, t_2$ are linearly independent over $\Q$.

For the same reason, there is bijection from $x_0 + \Lambda$ to $\Z^2$ as $x_0 + \lambda_1 t_1 + \lambda_2 t_2 \mapsto (\lambda_1, \lambda_2)$. 
Let $\ell\colon x_0 + \Lambda \to \mathbb{N}$ by $x \mapsto |\lambda_1| + |\lambda_2|$.  This map is well defined because of the bijection between $x_0 + \Lambda$ and $\Z^2$. 

 Recall now that if $\pi$ is a minimal valid function and $\pi =
 \tfrac{1}{2}(\pi^1 + \pi^2)$ with valid functions~$\pi^1,\pi^2$, then by
 Lemma~\ref{lem:minimality-of-pi1-pi2}, $\pi, \pi^1, \pi^2$ are minimal as
 well. Furthermore, from Lemma~\ref{lem:tightness}, if $\pi(x) +
 \pi(y) = \pi(x + y)$, then $\pi^i(x)+ \pi^i(y) = \pi^i(x + y)$ for $i=1,2$.
 Thus the difference $\bar\pi = \pi^1 - \pi$ satisfies also $\bar \pi(x)
 +\bar\pi(y) = \bar\pi(x + y)$.

\begin{lemma}
\label{lemma:barpiShift}
Suppose that for all $x \in [A, A_i]$, $\bar \pi(a_i) + \bar \pi(x) = \bar \pi(a_i + x)$ for $i=0,1,2$.
If $x = x_{0} + \lambda_1 t_1 + \lambda_2 t_2 \in (x_0 + \Lambda) \cap [A, A_0]$ with $\lambda_1, \lambda_2 \in \Z$, then 
$$
\bar \pi(x) - \bar \pi(x_0) = \lambda_1 \left(\bar \pi({a_1}) - \bar \pi({a_0})\right) + \lambda_2 \left(\bar \pi({a_2}) - \bar \pi({a_0})\right).
$$
\end{lemma}
\begin{proof}
Our proof will be by induction on $\ell(x)$.
For $\ell(x) = 0$, we have $\lambda_1 = \lambda_2 = 0$, thus $x = x_{0}$ and the result is trivial.

Now suppose $ \ell(x) > 0$.  We will show the result for $x \in (x_0,A_0]$ as the proof for $x \in [A, x_0)$ is similar.  Since $x > x_0$, we must have $\lambda_i > 0$ for either $i=1$ or $i=2$.  Without loss of generality, let $i=1$.  Note that $x \in (x_0,A_0] \subset [A,A_0]$, and therefore 
$$
\bpi(a_0) + \bpi(x) = \bpi(x + a_0).
$$
Consider the point $x - t_1 = x + a_0 - a_1 = x - A_0 + A_1$.  First, since $x \leq A_0$,  we have $x - t_1 \leq A_0 - A_0 + A_1 = A_1$.   
 Second,  since $x > x_0$, w have $x - t_1 > x_0 - t_1 \geq A$.  Therefore $x - t_1 \in [A, A_1]$ and 
$$
\bpi(a_1) + \bpi(x - t_1) = \bpi(x + a_0).
$$
Subtracting these two equations and rearranging terms yields 
$$
\bpi(x)  - \bpi(x-t_1)= \bpi(a_1) - \bpi(a_0).
$$ 
Since $x-t_1 \in [A,A_0]$ and $\ell(x-t_1) = |\lambda_1 - 1| + |\lambda_2| = |\lambda_1|-1 + |\lambda_2| = \ell(x) -1$,  the induction hypothesis holds for $x - t_1$ and hence
$$
\bar \pi(x-t_1)  - \bpi(x_0)= (\lambda_1-1) (\bar \pi(a_1) - \bar \pi({a_0})) + \lambda_2 (\bar \pi({a_2}) - \bar \pi({a_0})).
$$
Therefore
\begin{align*}
\bar \pi(x) - \bar \pi(x_0) 
&=   \big(\bpi(x) - \bar \pi(x-t_1)  \big)+ \big(\bar \pi(x-t_1) - \bpi(x_0)\big)\\
&= \big(\bpi({a_1}) - \bpi({a_0})  \big) + \big((\lambda_1-1) (\bar \pi({a_1}) - \bar \pi({a_0})) + \lambda_2 (\bar \pi({a_2}) - \bar \pi({a_0}))\big)  \\
&= \lambda_1 \left(\bar \pi({a_1}) - \bar \pi({a_0})\right) + \lambda_2 \left(\bar \pi({a_2}) - \bar \pi({a_0})\right).
\end{align*}
As stated before, the proof for $x \in [A,x_0)$ is similar and is done by supposing $\lambda_1 < 0$ and considering the point $x + a_1 - a_0 = x - t_1$.  The calculations are very similar. 
\end{proof}

\subsection{Construction}
We now give a precise definition of the function in Figure \ref{figure:irrationalFunction}, and then apply the above lemma.  This function will have three slopes, $c_1, c_2, c_3$, where we choose $c_2 = 0$.  The construction will be done in two steps.  We leave this construction general because, although we will only show it is extreme for one choice of parameters, it is indeed extreme for many choices. \smallbreak

\noindent\textbf{Step 1.}  We will determine variables $d_j^i$, signifying the $i^{\text{th}}$ interval of slope $c_j$.  The intervals written in order have the following lengths:
$$
d_1^1, d_3^1, d_1^2,  d_3^2, d_2^1, d_1^3, d_2^2,  d_3^3,  d_1^4, d_3^4, d_1^5, 1-f,
$$
where $a_0 = d_1^1 + d_3^1$, $A = a_0 + d_1^2 + d_3^2$, $A_0 = A + d_2^1$, and $f-A = A + d_2^1 + d_1^3$.  Note that the interval $[A, A_0]$ has slope $c_2 = 0$ in this notation.

We begin by picking $f \in (0,1)$.  On the interval $[f,1]$, the function will have only slope $c_3$.  We divide up the length of the interval $[0,f]$ into lengths $d_1, d_2, d_3$ that will be the amount of $[0,f]$ with slopes $c_1, c_2, c_3$, respectively. So $d_1 + d_2 + d_3 = f$.  Therefore we have the following equations:
\begin{align*}
c_1 = \frac{1 - d_2c_2 -  d_3 c_3}{d_1}, \ \ c_2 = 0,  \  \ c_3 = \frac{-1}{1-f},\ \ \ 
d_1 +  d_2 + d_3 = f.
\end{align*}
Now we subdivide each of these lengths into smaller lengths.  Inside $[0,f]$, we want 5 intervals with slope $c_1$, 2 intervals with slope $c_2$ and 4 intervals with slope $c_3$.  Therefore we give ourselves the following positive variables
\begin{align*}
d_1^1 + d_1^2 + d_1^3 + d_1^4 + d_1^5 = d_1,\\
d_2^1 + d_2^2 = d_2,\\
d_3^1 + d_3^2 + d_3^3 + d_3^4 = d_3.
\end{align*}
To preserve symmetry in the function, we restrict ourselves to 
\begin{align*}
d_1^1 = d_1^5, \ \ \
d_1^2 = d_1^4, \ \ \
d_2^1 = d_2^2, \ \ \
d_3^1 =  d_3^4, \ \ \
d_3^2 = d_3^3.
\end{align*}
We choose our parameters such that 
$$\pi(a_0) + \pi(A) = \pi(f-A) \ \  \ \text{ and }\ \ \ a_0 + A = f-A,$$
 which gives the two equations
\begin{align*}
d_1^1 + d_3^1 = d_2^1 + d_1^3,\ \ \ \ \ 
d_1^1 c_1 + d_3^1 c_3 = d_2^1 c_2 + d_1^3 c_3.
\end{align*}
By choosing values for the variables $f, d_1, d_2$, and $d_1^1 + d_3^1$, the remaining variables can be determined  uniquely.
\smallskip

\noindent\textbf{Step 2.}  We will create $2$ more additivity equations: 
$$\pi(a_1) + \pi(A) = \pi(a_1 + A)\ \ \  \text{  and  }\ \ \  \pi(a_2) + \pi(A) = \pi(A+ a_2).$$ 
For this, we pick positive $\delta^1, \delta^{2}$ such that the sum of these
values is less than $d_2^1/2$.   We set $t_1 = \delta^1$,  $t_2 = \delta^1 +
\delta^2$, $a_1 = a_0 + t_1$ and $a_2 = a_0 + t_2$.  For each $\delta^i$, find
$\delta^i_1, \delta^i_3$ such that $\delta^i_1 + \delta^i_3 = \delta^i$ and
$c_1 \delta^i_1 + c_3 \delta^i_3 = 0$.   With this, we force $\pi(a_0) =
\pi(a_1) = \pi(a_2)$, and the desired additivity relations follow.

Now we adjust some previous parameters to make room for new zig-zags:
\begin{align*}
\tilde d_1^2 = d_1^2 - \delta_1^1 - \delta_1^{2},\ \ \
\tilde d_1^4 = d_1^4 - \delta_1^1 -  \delta_1^{2},\ \ \
\tilde d_3^2 = d_3^2 - \delta_3^1 - \delta_3^{2},\ \ \
\tilde d_3^3 = d_3^3 - \delta_3^1 - \delta_3^{2}.
\end{align*}
Let $e_3 = 1-f$.
Now the intervals written in order have the following lengths:
$$
d_1^1, d_3^1, \delta_1^1, \delta_3^1, \delta_1^{2}, \delta_3^{2}, \tilde d_1^2, \tilde d_3^2, d_2^1, d_1^3, d_2^2, \tilde d_3^3, \tilde d_1^4,
\delta_3^{2}, \delta_1^{2}, \dots, \delta_3^{1}, \delta_1^{1},
d_3^4, d_1^5, e_3,
$$
where each interval has slope corresponding to the subscript.  Let $t_1 = \delta^1, t_2 = \delta^1  + \delta^2$, and let $a_1 = a_0 + t_1$, $a_2 = a_0 + t_2$.  The parameters $\delta^1, \delta^2$ will be chosen such that $t_1, t_2$ are linearly independent over $\Q$.

\subsection{Proof of extremality in the irrational case}

\begin{theorem}\label{thm:irrational-extreme}
Let $f = 4/5,  d_1 = 3/5, d_3 = 1/10$, $d_1^1   + d_3^1 = 15/100$, $\delta^1 = 1/200, \delta^2 = \sqrt{2}/200$ and let $\pi$ be the function given by these parameters and the construction above.  Then $\pi$ is extreme.
\end{theorem}
\begin{proof}
Using Theorem~\ref{minimality-check}, it can be shown that $\pi$ is a
minimal function.  Furthermore, these parameters satisfy Assumption~\ref{assum:requirements} and $\pi$ satisfies the hypotheses of Lemma \ref{lemma:barpiShift}.
Let $\pi^1, \pi^2$ be minimal valid functions such that $\pi = \frac12(\pi^1 + \pi^2)$.  From Theorem \ref{Theorem:functionContinuous}, it follows that $\bar\pi = \pi^1 - \pi$ is continuous, and since $\pi, \pi^1, \pi^2$ are minimal, $\bar \pi(0) = 0$.  

Let $I\in \I_{B,\edge}$ be an interval with slope $c_1$ and let $x \in
\intr(I)$.  Since $\intr(I)$ is open, there exists an $\epsilon > 0$ such that
$U = [x - \epsilon/2, x + \epsilon/2] \subset \intr(I)$ and $V = [0,\epsilon]
\subset [0,d_1^1]$.   Since $\pi(u) + \pi(v) = \pi(u+v)$ for all $u \in U,
v\in V$, by the Interval Lemma, $\pi$ is affine imposing in $U,V,U+V$, 
and $\bar\pi$ has the same slopes on~$U,V,U+V$.  
Repeating this for every point $x$ where $\pi$ has
slope $c_1$ shows that $\pi$ is affine imposing in all intervals where it has slope
$c_1$ and $\bar\pi$ has the same slope, say $\bar c_1$, in these intervals. 
In other words, $\bar \pi'(x) = \bar c_1$ for all
$x$ such that $\pi'(x) = c_1$. 

 Similarly, we find that $\pi$ is affine imposing in all intervals with slope
 $c_3$, and the function~$\bar\pi$ has the same slopes in
 these intervals.  Furthermore, since $\pi$ has slope $c_3$ on the interval $[f,1]$ and $\theta(f) = 1$, $\theta(1) = 0$ for $\theta = \pi,\pi^1$, the slope $c_3$ is fixed for each function to $c_3 = \smash[b]{\tfrac{-1}{1-f}}$.  Therefore $\bar \pi'(x) = 0$ for all $x$ such that $\pi'(x) = c_3$.

 Since $\bar \pi$ is a continuous piecewise linear function and $\bar \pi(0) = 0$, $\bar \pi(x) = \int_0^x \bar \pi'(t)\,\mathrm{d}t$, where the finitely many values of $t$ where $\bar \pi'$ does not exist are ignored.  The following calculations stem from the facts that $\bar \pi'(x) = 0$ wherever $\pi'(x) = c_3$ and that $\pi'$ on $[0, a_2]$ is either of slope $c_1$ or $c_3$.  Therefore, 
\begin{align*}
\bar\pi(a_1) - \bar \pi(a_0) &= \bar c_1 \delta_1^1,\\
\bar\pi(a_2) - \bar \pi(a_0) &= \bar c_1 (\delta_1^1 + \delta_1^2).
\end{align*}
Note that $\delta_1^i  = (1 - c_1/c_3) \delta^i$.  Let $\alpha = 1-c_1/c_3$.  Let $t_1 = \delta^1$ and $t_2 = \delta^1 + \delta^2$. Then
\begin{align*}
\bar\pi(a_1) - \bar \pi(a_0) = \bar c_1 \alpha t_1,\\
\bar\pi(a_2) - \bar \pi(a_0) = \bar c_1 \alpha t_2.
\end{align*}

Recall that $[A, A_0]$ is the first interval with slope $c_2=0$ and let $x \in (x_0 + \Lambda) \cap [A, A_0]$.  By Lemma \ref{lemma:barpiShift}, we have
\begin{align*}
\bar \pi(x) - \bar \pi(x_0) &=\lambda_1 (\bar \pi({a_1}) - \bar \pi({a_0}) + \lambda_2 (\bar \pi({a_2}) - \bar \pi({a_0}))\\
&=   \lambda_1 \bar c_1 \alpha t_1 + \lambda_2 \bar c_1 \alpha t_2\\
&= \alpha \bar c_1 (\lambda_1 t_1 + \lambda_2 t_2)\\
& =\alpha \bar c_1(x - x_0).
\end{align*}
That is 
$$
\bar \pi(x) = \bar \pi(x_0) + \alpha \bar c_1 (x - x_0)
$$
and therefore, since $\Lambda$ is dense in $\R$, $\bar \pi(x)$ is affine on a
dense set in $[A, A_0]$.  Since $\bar \pi $ is continuous and $\bar \pi$ is
affine on all of $[A, A_0]$,  the function  $\pi$ is affine imposing in $[A, A_0]$.  Since $[A, A_0]$ and $f-[A, A_0]$ are connected via the symmetry reflection, $\pi$ is also affine imposing in $f-[A, A_0]$ and these intervals must have the same slope.

Since $\pi$ is continuous with three slopes, we can set up a system of equations on three slopes that is satisfied by $\pi, \pi^1, \pi^2$.   
We will demonstrate this system has a unique solution using just the equations $\pi(f) = 1$, $\pi(1) = 0$, and $\pi(a_0) + \pi(A) = \pi(f-A)$ where $a_0,A,f-A$ are defined above.  These equations yield the following system of equations, which is invariant with respect to $\delta^1, \delta^{2}$:

\begin{equation}
\begin{bmatrix}
d_1 & d_2 & d_3 \\
d_1 & d_2 & d_3 + e_3\\
d_1^1- d_1^3 & - d_2^1 &  d_3^1
\end{bmatrix}
\begin{bmatrix}
c_1 \\ c_2 \\ c_3
\end{bmatrix}
 =
 \begin{bmatrix}
 1 \\ 0 \\ 0
 \end{bmatrix}.
\end{equation}
Plugging the parameters into the system of equations, we have the matrix equation
\begin{equation}
\begin{bmatrix}
3/5 & 1/10 & 1/10\\
3/5 & 1/10& 3/10\\
2/15 - 1/10 & - 1/20 &  1/60
\end{bmatrix}
\begin{bmatrix}
c_1 \\ c_2 \\ c_3
\end{bmatrix}
 =
 \begin{bmatrix}
 1 \\ 0 \\ 0
 \end{bmatrix}.
\end{equation}
This system has a unique solution, which is $c_1 = 5/2, c_2 = 0, c_3 = -5$.  Therefore, $\pi = \pi^1 = \pi^2$ and $\pi$ is extreme.
\end{proof}

\subsection{Proof of non-extremality in the rational case}

\begin{theorem}\label{thm:rational-non-extreme}
Let $f = 4/5,  d_1 = 3/5, d_3 = 1/10$, $d_1^1   + d_3^1 = 15/100$.  Under the construction above, these values uniquely determine the values
$
d_1^1, d_3^1, d_1^2,  d_3^2, d_2^1, d_1^3, d_2^2,  d_3^3,  d_1^4, d_3^4, d_1^5.$   Let $0<\delta^1, \delta^2\in\mathbb{Q}$ be any rational numbers such that $\delta^1 +\delta^2 < d^2_1 + d^2_2$ and let $\pi$ be the function given by these parameters and the construction above.  Then $\pi$ is not extreme.
\end{theorem}
\begin{proof}
  Since the choice of $\delta^1, \delta^2$ does not affect the slope on the
  intervals, we know from the proof of Theorem~\ref{thm:irrational-extreme} that the
  slopes of $\pi$ are $c_1 = 5/2, c_2 = 0, c_3 = -5$.  
  By the construction, the function $\pi$ takes slope $c_2=0$ only on two
  intervals, $[A,A_0]$ and $f-[A,A_0]$, which are symmetric about the point
  $f/2$.  Between them is the interval $[A_0, f-A_0]$ of length $d^3_1 =
  1/10$, on which $\pi$ has slope $c_1 = 5/2$.   The midpoint of this interval
  is $f/2$.  Since $\pi$ satisfies the symmetry condition, $\pi(f/2) = 1/2$.  
  We thus find
  \begin{subequations}
    \label{eq:pi-flat-values}
    \begin{align}
      \pi(x) &= \frac12 - \frac12 \cdot \frac1{10} \cdot
      \frac52 = \frac12 - \frac18 = \frac38 \quad \text{for} \quad x\in[A,A_0]\\
      \intertext{and}
      \pi(x) &= \frac12 +  \frac12 \cdot \frac1{10} \cdot \frac52 = \frac12 +
      \frac18 = \frac58 \quad \text{for} \quad x\in f-[A,A_0]
    \end{align}
  \end{subequations}
  as the values of~$\pi$ on the two slope-$0$ intervals.
  
  Since all the input parameters were chosen to be rational, there exists a
  number $q \in \Z_+$ such that the breakpoints of $\pi$ are all in $\frac1q
  \Z$.  We consider the graph $\G = \G(\Iedge/\Z,\E)$ introduced in
  section~\ref{section:AI}.  By Lemma~\ref{lemma:pointLemma}, if the
  equivalence classes $\EqClass{I}, \EqClass{J}$ of two
  intervals $I,J\in \Iedge$ are connected by a path in~$\E$, then $\pi$ has
  the same slope on $I$ and~$J$.  Let $\mathcal Z \subseteq \Iedge$ be the set
  of all intervals where $\pi$ has slope $c_2 = 0$. Then, in particular,
  the intervals $I\in \mathcal Z$ are connected only to other intervals
  in~$\mathcal Z$. 

  Now we will show that none of the intervals in~$\mathcal Z$ is covered,
  i.e., $\mathcal Z \cap \P^2_{q,\edge} = \emptyset$.
  Suppose $I\in\mathcal Z$ is covered. 
  By Lemma~\ref{lemma:patching-interval-lemma}, 
  there exists a two-dimensional face $F \in \Delta\P_q$ 
  with 
  \begin{equation}\label{eq:delta-pi-is-zero}
    \Delta\pi(x,y) = \pi(x) + \pi(y) - \pi(x+
    y) = 0 \text{ for all } (x,y) \in \intr(F).
  \end{equation}
  such that $I$ is one of the intervals $I_1=p_1(F)$, $I_2=p_2(F)$, or
  $I_3=p_3(F)$, and $\pi$ has 
  the same slope on $I_1$, $I_2$, $I_3$.  Thus $\pi$ has slope~$0$ on 
  $I_1$, $I_2$, $I_3$. 
  By~\eqref{eq:pi-flat-values}, however, $\pi(x), \pi(y), \pi(x + y)
  \equiv \frac18 \pmod{\frac14}$ and thus $\Delta\pi(x,y)
  \equiv\frac18\pmod{\frac14}$ for $(x,y) \in \relint(F)$. This is a contradiction
  to~\eqref{eq:delta-pi-is-zero}. 
  
  Thus we have showed that none of the equivalence classes of the intervals in~$\mathcal Z$ is
  connected by a path in~$\G$ to the equivalence class of a covered interval. 
  Therefore $\S^2_{q,\edge} \neq \Iedge$.
  Hence, by Lemma~\ref{lemma:not-extreme}, $\pi$ is not extreme.
\end{proof}

Theorems \ref{thm:irrational-extreme} and \ref{thm:rational-non-extreme} show
that a piecewise linear function with irrational breakpoints can be extreme,
even if any nearby piecewise linear function with rational breakpoints is not
extreme.  This may suggest that an algorithm for deciding the extremality of
functions with rational breakpoints that is oblivious to the sizes of the
denominators of the breakpoints is not possible.

\providecommand\ISBN{ISBN }
\bibliographystyle{../amsabbrvurl}
\bibliography{../bib/MLFCB_bib}

\end{document}